\documentclass[11pt]{amsart}

\usepackage{amssymb,esint}

\newtheorem{teo}{Theorem}[section]
\newtheorem{lm}[teo]{Lemma}
\newtheorem{prop}[teo]{Proposition}
\newtheorem{coro}[teo]{Corollary}
\newtheorem*{teoA}{Theorem A}
\newtheorem*{teoB}{Theorem B}
\theoremstyle{definition}

\newtheorem{oss}[teo]{Remark}
\newtheorem*{ack}{Acknowledgements}

\numberwithin{equation}{section}

\author[Bousquet]{Pierre Bousquet}
\address{Aix-Marseille Universit\'e, CNRS, Centrale Marseille, I2M, UMR 7373, 13453 Marseille, France}
\email{pierre.bousquet@univ-amu.fr}

\author[Brasco]{Lorenzo Brasco}
\address{Aix-Marseille Universit\'e, CNRS, Centrale Marseille, I2M, UMR 7373, 13453 Marseille, France}
\email{lorenzo.brasco@univ-amu.fr}

\author[Julin]{Vesa Julin}
\address{Department of Mathematics and Statistics, University of Jyv\"askyl\"a, P.O. Box 35 (MaD), 40014 Jyv\"askyl\"a, Finland}
\email{vesa.julin@jyu.fi}


\date{\today}

\title[Widely degenerate problems]{Lipschitz regularity for local minimizers\\ of some widely degenerate problems}

\begin{document}

\begin{abstract}
We consider local minimizers of the functional 
\[
\sum_{i=1}^N \int (|u_{x_i}|-\delta_i)^p_+\, dx+\int f\, u\, dx,
\]
where $\delta_1,\dots,\delta_N\ge 0$ and $(\,\cdot\,)_+$ stands for the positive part.
Under suitable assumptions on $f$, we prove that local minimizers are Lipschitz continuous functions if $N=2$ and $p\ge 2$, or if $N\ge 2$ and $p\ge 4$. 
\end{abstract}

\maketitle

\tableofcontents

\section{Introduction}

\subsection{Overview}
This paper is devoted to prove Lipschitz continuity for local minimizers of the anisotropic functional 
\begin{equation}
\label{Fintro}
\mathfrak{F}(u;\Omega')=\sum_{i=1}^N \int_{\Omega'} \frac{(|u_{x_i}|-\delta_i)_+^p}{p}\,dx+\int_{\Omega'} f\, u\,dx,\qquad  u\in W^{1,p}_{loc}(\Omega),\ \Omega'\Subset\Omega.
\end{equation}
Here $\Omega\subset\mathbb{R}^N$ is an open set, $2\le p<\infty$, $\delta_i\ge 0$, $(\,\cdot\,)_+$ stands for the positive part and $f\in L^{p'}_{loc}(\Omega)$ where $p'=p/(p-1)$. 
This functional $\mathfrak{F}$ stands for a model case of a more general class of problems, with specific growth and monotonicity assumptions. For the sake of clarity, the results in this paper are only stated for $\mathfrak{F}$. However, their proofs can be easily adapted to embrace
general functionals having a similar structure.
\par
The functional $\mathfrak{F}$ naturally arises in problems of Optimal Transport with congestion and anisotropic effects, see for example \cite{BraCar_dyn,BraCar} for some motivations. These two papers contained among others some regularity results for local minimizers of \eqref{Fintro}. For instance \cite[Main Theorem]{BraCar} proves that if $f\in L^\infty_{loc}(\Omega)$, then $u$ is ``almost Lipschitz'', i.e. $u\in W^{1,r}_{loc}(\Omega)$ for every $r\ge 1$. On the other hand, in \cite{BraCar_dyn} it is proved that if $f\in W^{1,p'}_{loc}(\Omega)$, then
\begin{equation}
\label{sobolev}
(|u_{x_i}|-\delta_i)_+^\frac{p}{2}\, \frac{u_{x_i}}{|u_{x_i}|}\in W^{1,2}_{loc}(\Omega),\qquad i=1,\dots,N.
\end{equation}
However, it must be mentioned that to the best of our knowledge, {\it Lipschitz regularity of local minimizers is still unknown}. More surprisingly, even the case $\delta_1=\dots=\delta_N=0$ does not seem to be fully understood. Only very recently some results have been obtained in this case, see \cite{BD,De}.
\par
Observe that local minimizers of \eqref{Fintro} are local weak solutions of the anisotropic degenerate equation
\begin{equation}
\label{equazione}
\sum_{i=1}^N \left((|u_{x_i}|-\delta_i)_+^{p-1}\, \frac{u_{x_i}}{|u_{x_i}|}\right)_{x_i}=f,
\end{equation}
which reduces to the Poisson equation for the so-called {\it pseudo $p-$Laplacian} when $\delta_1=\dots=\delta_N=0$, i.e.
\begin{equation}
\label{equazione2}
\sum_{i=1}^N \left(|u_{x_i}|^{p-2}\, u_{x_i}\right)_{x_i}=f.
\end{equation}
The terminology ``pseudo $p-$Laplacian'' appears in \cite{BK}. We just point out that such an operator already appeared in J.-L. Lions's monograph \cite{Lions}, where existence issues for solutions to evolutions equations are tackled.  
\vskip.2cm
In order to neatly explain the difficulty of the problem, we now recall some class of functionals for which the Lipschitz property for local minimizers is known to be true.
The first one is given by 
\begin{equation}
\label{plaplace}
\int G(\nabla u)\, dx,
\end{equation}
with $G$ enjoying a {\it $p-$Laplacian-type structure at infinity}. This means that there exist $c,C>0$ and $m\ge 0$ such that $G$ verifies the ellipticity condition
\begin{equation}
\label{elliptic}
\langle D^2 G(z)\,\xi,\xi\rangle\ge c\,|z|^{p-2}\, |\xi|^2,\qquad |z|\ge m, 
\end{equation}
and the growth condition
\begin{equation}
\label{growth}
|D G(z)|\le C\, |z|^{p-1},\qquad |z|\ge m.
\end{equation}
We refer the reader to \cite{Br,CCG,CE,EMT} and \cite{FFM} for example. For completeness, we also point out the papers \cite{CF2,CF} and \cite{SV} for related regularity results on the term $\nabla G(\nabla u)$, when $m>0$.
\par
Another type of well-studied functionals having some similarities with $\mathfrak{F}$ 
is given by (see for example \cite{BFZ2, BFZ} and \cite[Section 4]{FS}) 
\begin{equation}
\label{splittingtype}
\int \widetilde G(\nabla u)\, dx,\qquad\qquad \mbox{ with }\quad \widetilde G(z)=\sum_{i=1}^N (\mu+|z_{i}|^2)^\frac{p_i}{2}.
\end{equation}
Here $\mu>0$ and $1<p_1\le p_2\le \dots\le p_N$ are possibly different exponents. When the $p_i$ are not equal, such a functional belongs to the class of {\it problems with non standard growth conditions}, whose systematic study started with the paper \cite{Ma89} by Marcellini. In this case we can infer local Lipschitz continuity if the exponents $p_i$ are not ``too far apart'' (see the above mentioned references for more details).
\par
However, our functional $\mathfrak{F}$ {\it does not fall neither in the class of the functional \eqref{plaplace} nor in that of \eqref{splittingtype}}. Indeed, observe that in our case
\[
F(z)=\sum_{i=1}^N \frac{(|z_i|-\delta_i)^p_+}{p},
\] 
verifies \eqref{growth}, {\it but \eqref{elliptic} crucially fails to hold}, since for every $m>0$, there always exists $z$ such that $|z|=m$ and the least eigenvalue of $D^2 F(z)$ is $0$. Observe that this phenomenon already occurs for the pseudo $p-$Laplacian, i.e. when $\delta_1=\dots=\delta_N=0$. Indeed, the main difficulty of the problem is that the region where ellipticity fails {\it is unbounded}.
\par
For the same reason, $\mathfrak{F}$ is not of the type \eqref{splittingtype}, since already in the standard growth case $2\le p_1=p_2=\dots=p_N$ we have
\[
0<\min_{|\xi|=1}\langle D^2 \widetilde G(z)\, \xi, \xi\rangle,\qquad z\in\mathbb{R}^N.
\]
When one allows $\mu=0$ in \eqref{splittingtype}, the corresponding functional becomes degenerate along the axes $z_i=0$, like in the case of the pseudo $p-$Laplacian. This case has been considered in the pioneering paper \cite{UU} by Uralt'seva and Urdaletova. There the Lipschitz character of minimizers has been shown under some restrictions on the exponents $p_1,\dots,p_N$, by using the so-called {\it Bernstein method}. Though the growth conditions considered are more general than ours, the type of degeneracy is again weaker than that admitted in $\mathfrak{F}$ (see the next subsection for more comments on the result of \cite{UU}).
\vskip.2cm
About the restriction $p\ge 2$ considered in this paper, it is noteworthy to observe that for $1<p<2$ our functional has a $p-$Laplacian-type structure. Indeed, in this case $p-2<0$ and thus \eqref{elliptic} is satisfied with $m=0$, i.e.
\[
\langle D^2 F(z)\,\xi,\xi\rangle=(p-1)\,\sum_{i=1}^N (|z_i|-\delta_i)_+^{p-2}\, |\xi_i|^2\ge (p-1)\, |z|^{p-2}\, |\xi|^2, \qquad \xi\in\mathbb{R}^N,z\in\mathbb{R}^N,
\]
while of course
\[
|DF(z)|\le |z|^{p-1},\qquad z\in\mathbb{R}^N.
\]
Then in this case local minimizers are locally Lipschitz continuous by\footnote{To be more precise, for $1<p< 2$ the function $F$ is not $C^2$. However, this is not an issue, since the result of \cite[Theorem 2.7]{FFM} holds for convex functions satisfying a qualified form of uniform convexity for $|z|\ge m$. This coincides with \eqref{elliptic} if the function is $C^2$, but it is otherwise more general.} \cite[Theorem 2.7]{FFM}. 

\subsection{Main results}
In this paper, we prove the following results.

\begin{teoA}[Two dimensional case]
\label{teo:A}
Let $N=2$ and $p\ge 2$. Let $f\in W^{1,p'}_{loc}(\Omega)$, where $p'=p/(p-1)$. Then every local minimizer $U\in W^{1,p}_{loc}(\Omega)$ of the functional $\mathfrak{F}$ is a locally Lipschitz continuous function.
\end{teoA}

\begin{teoB}[Higher dimensional case]
\label{teo:B}
Let $N\ge 2$ and $p\ge 4$. Let $f\in W^{1,\infty}_{loc}(\Omega)$. Then every local minimizer $U\in W^{1,p}_{loc}(\Omega)$ of the functional $\mathfrak{F}$ is a locally Lipschitz continuous function.
\end{teoB}
Let us now spend some words about the methods of proofs. The preliminary step in both cases is an approximantion argument. Namely, the functional $\mathfrak{F}$ is replaced by a regularized version $\mathfrak{F}_\varepsilon$, for a small parameter $\varepsilon>0$. This permits to infer the necessary regularity on the solutions $u_\varepsilon$ of the regularized problem, in order to justify the manipulations needed to obtain a priori Lipschitz estimates uniform in $\varepsilon$.
Then one aims at taking these estimates to the limit as $\varepsilon$ goes to $0$. However, one should pay attention to the fact that $\mathfrak{F}$ is not strictly convex when at least one $\delta_i\not =0$. Thus a sequence of solutions $u_\varepsilon$ may not necessarily converge to the selected local minimizer. In \cite{BraCar} a penalization argument was used to fix this issue. Here on the contrary, we use a simpler argument, based on the fact that the lack of strict convexity of $t\mapsto (|t|-\delta_i)_+^p$ is ``confined'' (see Lemma \ref{propagationofregularity}).
\vskip.2cm
The core of the proof of Theorem A is the a priori Lipschitz estimate of Proposition \ref{lm:2d}. Such an estimate is achieved by means of a Moser's iteration technique applied to the equation solved by the partial derivatives $u_{x_j}$ of the local minimizer. More precisely, we look at power-type subsolutions  of this equation, i.e. quantities like $|u_{x_j}|^s$ for $s\ge 1$.
This is a standard strategy for equations having a $p-$Laplacian-type structure, but as already said our operator does not have such a structure and this entails several additional difficulties.
\par
As explained in the introduction of \cite{BraCar}, the main difficulty of this method is that the Caccioppoli inequality we get for $|u_{x_j}|^s$ is quite involved.
Indeed, due to the particular structure of $D^2 F$, in principle we have a control only on a ``weighted'' norm of $\nabla |u_{x_j}|^s$, the weights being dependent on {\it all the other components $u_{x_i}$ of the gradient} (see Lemma \ref{lm:caccioweight} below). Roughly speaking, what we control in the Caccioppoli inequality is a quantity like
\[
\sum_{i=1}^N \int |u_{x_i}|^{p-2}\, \left|\left(|u_{x_j}|^{s+1}\right)_{x_i}\right|^2.
\]
For the diagonal term, i.e. when $i=j$, we can combine the $x_j-$derivative of $u_{x_j}$ with the weigth $|u_{x_j}|^{p-2}$ and simply recognize the $x_j-$derivative of yet another power of $u_{x_j}$. Since we would like to have a control on the full gradient of such a power of $u_{x_j}$, we still miss all the $x_i-$derivatives ($i\not=j$) of this function. To overcome this difficulty, we use in a crucial way the Sobolev property \eqref{sobolev} together with H\"older inequality, in order to ``cook-up'' suitable Caccioppoli inequalities for all these {\it missing terms}. Surprisingly enough, even if the functional $\mathfrak{F}$ has $p-$growth in every direction, we rely on the {\it anisotropic Sobolev inequality} due to Troisi (see \cite{Tr}) in order to produce an iterative scheme of reverse H\"older inequalities.
This procedure works for $N=2$, but it seems to be limited just to the two dimensional case (see Remark \ref{oss:2d?} below).
\vskip.2cm
In contrast Theorem B is valid in every dimension, but we need the restriction $p\ge 4$. 
This second result partially superposes with the already mentioned \cite[Theorem 1]{UU} by Uralt'seva and Urdaletova. However, it should be noticed that the monotonicity assumptions on the operator\footnote{See equation (8) of the paper \cite{UU}.} made in \cite{UU} does not allow for $\delta_i>0$. Moreover, the result in \cite{UU} is stated for $p>3$, but a careful inspection of the proof reveals that the same condition $p\ge 4$ is needed there as well\footnote{This comes from hypothesis (5) in \cite{UU}. Also observe that this condition contains a small typo, $m_{i-2}$ should be replaced by $m_i-2$.}. 
\par
Both the proofs of Theorem B and that of \cite[Theorem 1]{UU} are based on a priori Lipschitz bounds, obtained by means of pointwise estimates in the vein of Bernstein method.
However, computations are not the same and we believe ours to be slightly simpler. In \cite{UU} the first step is to look at the equation solved by a {\it concave} power of $u$, given by the function
\[
w=(u+\|u\|_{L^\infty}+1)^\gamma,\qquad 0<\gamma<1.
\] 
Then they consider the equation solved by (some function of) $\nabla w$. There is an extra term in this new equation coming from the concave power which crucially leads to the result.
\par 
Here on the contrary we obtain the Lipschitz estimate by directly attacking equation \eqref{equazione}. The main point is to consider the equation satisfied by the quantity 
\[
|\nabla u|^2+\lambda\, u^2,
\]
for a suitably large paramater $\lambda$. We notice that this is exactly the same test function used to prove classical gradient estimates for {\it linear} uniformly elliptic equations (see for example \cite[Proposition 2.19]{HL}).  
\vskip.2cm
One of the drawbacks of these two strategies is the assumption on $f$, which does not seem to be optimal. Indeed, we expect the result to be true under the natural hypothesis $f\in L^{q}_{loc}(\Omega)$ with $q>N$.

\subsection{Plan of the paper} In Section \ref{sec:preliminaries} we set notations and preliminary results needed throughout the whole paper. In particular, we introduce there a regularized version of the problem which will be useful in order to get the desired Lipschitz estimate.
Then Section \ref{sec:caccioppoli} is devoted to prove some Caccioppoli-type inequalities for the gradient of the solution of the regularized problem. The proof of Theorem A is contained in \ref{sec:2d}, while Section \ref{sec:d} contains the proof of Theorem B. Two appendices containing some technical results complement the paper.

\begin{ack}
%The authors gratefully acknowledge useful conversations with Giovanni Cupini, Guido De Philippis, Nicola Fusco, Tuomo Kuusi, Paolo Marcellini and Giuseppe Mingione. 
A quick but stimulating discussion with Nina Uralt'seva in June 2012 led to a better understanding of the paper \cite{UU}, we thank her. Guillaume Carlier is warmly thanked for his interest in this work. Part of this paper has been written during the conferences ``{\it Journ\'ees d'Analyse Appliqu\'ee Nice-Toulon-Marseille}'' held in Porquerolles in May 2014, ``{\it Nonlinear partial differential equations and stochastic methods}'' held in Jyv\"askyl\"a in June 2014 and ``{\it Existence and Regularity for Nonlinear Systems of Partial Differential Equations}'' held in Pisa in July 2014. Organizers and hosting institutions are gratefully acknowledged.
\end{ack}

\section{Preliminaries}
\label{sec:preliminaries}

\subsection{Definitions and basic results}
Let $\Omega\subset\mathbb{R}^N$ be an open set and $p\ge 2$. In what follows we set for simplicity
\[
g_i(t)=\frac{1}{p}\, (|t|-\delta_i)^p_+,\qquad t\in\mathbb{R},\ i=1,\dots,N,
\]
where $0\le \delta_1,\dots,\delta_N$ are given real numbers. We will also define
\begin{equation}
\label{delta}
\delta=1+\max\{\delta_i\, :\, i=1,\dots,N\}.
\end{equation}
\begin{oss}[Smoothness of $g_i$]
When \(p\) is an integer and $\delta_i>0$, \(g_i\) is of class \(C^{p-1,1}\). When \(p\not\in \mathbb{N}\), then \(g_i\in C^{[p], p-[p]}(\mathbb{R})\)  where \([\,\cdot\,]\) denotes the integer part. 
\end{oss}
\begin{oss}[The limit case $p=2$]
Observe that for $p=2$ and $\delta_i>0$, we have $g_i\in C^{1,1}(\mathbb{R})\cap C^\infty(\mathbb{R}\setminus\{1,-1\})$, but $g_i\not\in C^2(\mathbb{R})$. In this case, like in \cite{BraCar} a smoothing around $|t|=\delta_i$ would be necessary, notably for the result of Lemma \ref{lm:approximation} below. However, in order not to overburden the presentation, for the sequel we will assume for simplicity $p>2$ (see \cite[Section 2]{BraCar} for more details).
\end{oss}

We are interested in local minimizers of the following variational integral
\begin{equation}
\label{F}
\mathfrak{F}(u;\Omega')=\sum_{i=1}^N \int_{\Omega'} g_i(u_{x_i})\, dx+\int_{\Omega'} f\, u\, dx,\qquad u\in W^{1,p}_{loc}(\Omega),
\end{equation}
where $f\in L^{p'}_{loc}(\Omega)$ and $\Omega'\Subset\Omega$. We recall that $u\in W^{1,p}_{loc}(\Omega)$ is said to be a {\it local minimizer} of $\mathfrak{F}$ if for every $\Omega'\Subset\Omega$ we have
\[
\mathfrak{F}(u;\Omega')\le \mathfrak{F}(u+\varphi;\Omega'),\qquad \mbox{ for every }\ \varphi\in W^{1,p}_0(\Omega').
\]
We first observe that $\mathfrak{F}$ is not strictly convex, unless $\delta=1$.
Thus minimizers are not unique in general. The following result guarantees that it will be sufficient to prove the desired result for one minimizer.
\begin{lm}[Propagation of regularity]
\label{propagationofregularity}
Let $B\Subset\Omega$ be a ball and $\varphi\in W^{1,p}(B)$. Let $u_1,u_2\in W^{1,p}(\Omega)$ be two solutions of
\begin{equation}
\label{mininitial}
\min\left\{\mathfrak{F}(v;B)\, :\, v-\varphi\in W^{1,p}_0(B)\right\}.
\end{equation}
Then it holds
\begin{equation}
\label{stimina}
\Big||(u_1)_{x_i}|-|(u_2)_{x_i}|\Big|\le 2\,\delta_i,\qquad \mbox{ a.\,e. in } B,\ i=1,\dots,N.
\end{equation}
In particular, if a minimizer of \eqref{mininitial} is (locally) Lipschitz, then this remains true for all the other minimizers.
\end{lm}
\begin{proof}
Let us suppose that \eqref{stimina} is not true. Then there exists $i_0\in\{1,\dots,N\}$ such that
\[
E_{i_0}:=\left\{x\in B\, :\, \Big||(u_1)_{x_{i_0}}|-|(u_2)_{x_{i_0}}|\Big|> 2\, \delta_i\right\},
\]
has strictly positive measure. We then set $u_s=(1-s)\, u_0+s\, u_1$ for some $s\in(0,1)$ and observe that this is admissible in \eqref{mininitial}. In view of Lemma \ref{lm:convessofuori} in Appendix A,
\[
g_{i_0}\left((1-s)\,(u_1)_{x_{i_0}}+s\, (u_2)_{x_{i_0}}\right)<(1-s)\,g_{i_0}((u_1)_{x_{i_0}})+s\,g_{i_0}((u_2)_{x_{i_0}}),\quad \mbox{ a.\,e. in} E_{i_0}.
\]
Thus we get 
\[
\mathfrak{F}(u_s)<(1-s)\,\mathfrak{F}(u_1)+s\,\mathfrak{F}(u_2)=\mathfrak{F}(u_1)=\mathfrak{F}(u_2),
\]
which gives the desired contradiction.
\end{proof}
We will also need the following regularity result, which is reminiscent of \cite{St}. A more general result of this type can be found in \cite{BB}.
\begin{teo}[\cite{BC}]
\label{teo:boucla}
Let $B\subset\mathbb{R}^N$ be a ball, $\varphi\in C^2(\overline B)$ and $f\in L^\infty(B)$. Let $u\in W^{1,1}(B)\cap L^\infty(B)$ be a solution of
\[
\min\left\{\int_B H(\nabla v)\, dx+\int_B f\, v\, dx\, :\, v-\varphi\in W^{1,1}_0(B)\right\},
\]
where $H:\mathbb{R}^N\to [0,\infty)$ is a $C^2$ convex function such that  for some \(\mu>0\)
\begin{equation}
\label{dessous}
\langle D^2\, H(z)\, \xi,\xi\rangle\ge \mu\, |\xi|^2,\qquad \xi,z\in\mathbb{R}^N.
\end{equation}
Then $u\in W^{1,\infty}_{loc}(B)$. 
\end{teo}

\subsection{Approximation scheme}

We now introduce a regularized version of the original problem. We set
\begin{equation}
\label{gepsilon}
g_{i.\varepsilon}(t)=g_i(t)+\frac{\varepsilon}{2}\, t^2=\frac{1}{p}\, (|t|-\delta_i)_+^{p}+\frac{\varepsilon}{2}\, |t|^2,\qquad t\in\mathbb{R}.
\end{equation}
Let $U$ be a local minimizer of $\mathfrak{F}$. We also fix a ball 
\[
B \Subset \Omega\quad \mbox{ such that }\quad 2\,B\Subset\Omega \mbox{ as well}.
\] 
Here $2\,B$ denotes the ball having the same center as $B$ scaled by a factor $2$.
\par
For every $0<\varepsilon\ll 1$ and every \(x\in \overline{B}\), we set $U_\varepsilon(x)=U\ast \varrho_\varepsilon(x)$, where $\varrho_\varepsilon$ is a smooth convolution kernel, supported in a ball of radius $\varepsilon$ centered at the origin. 
\par
Then by definition of $U_\varepsilon$ there exists $0<\varepsilon_0<1$ such that for every $0<\varepsilon\le \varepsilon_0$
\begin{equation}
\label{Uepsilon}
\|U_\varepsilon\|_{W^{1,p}(B)}= \|\nabla U_\varepsilon\|_{L^p(B)}+\|U_\varepsilon\|_{L^p(B)}\le \|\nabla U\|_{L^p(2\,B)}+\|U\|_{L^p(2\,B)}=:C_1.
\end{equation}
Finally, we define
\[
\mathfrak{F}_\varepsilon(v;B)=\sum_{i=1}^N \int_B g_{i,\varepsilon}(v_{x_i})\, dx+\int_B f_{\varepsilon}\, v\, dx,
\]
where \(f_{\varepsilon}=f\ast\varrho_{\varepsilon}\). The following preliminary result is standard.
\begin{lm}[Basic energy estimate]
There exists a  unique solution $u_\varepsilon$ to the problem 
\begin{equation}
\label{approximated}
\min\left\{\mathfrak{F}_\varepsilon(v;B)\, :\, v-U_\varepsilon\in W^{1,p}_0(B)\right\}.
\end{equation}
The following uniform energy estimate holds
\begin{equation}
\label{uniformeg}
\int_B |\nabla u_\varepsilon|^p\, dx\le C_2,
\end{equation}
for some constant $C_2=C_2(N,\delta,p,|B|,C_1,\|f\|_{L^{p'}(2\,B)})>0$. 
\end{lm}
\begin{proof}
We start by observing that a solution $u_\varepsilon$ exists, by a standard application of the Direct Methods. Uniqueness then follows from strict convexity of the integrand
\begin{equation}
\label{integrand}
L_\varepsilon(x,u,z)=\sum_{i=1}^N g_i(z_i)+\frac{\varepsilon}{2}\, |z|^2+f_{\varepsilon}(x)\, u,
\end{equation}
in the gradient variable. 
In order to prove \eqref{uniformeg}, we use the minimality of $u_\varepsilon$, which implies \(\mathfrak{F}_{\varepsilon}(u_{\varepsilon};B)\le \mathfrak{F}_{\varepsilon}(U_{\varepsilon};B)\). This gives
\[
\sum_{i=1}^N \int_B g_{i,\varepsilon}((u_{\varepsilon})_{x_i})\,dx \leq \sum_{i=1}^N \int_B g_{i,\varepsilon}((U_{\varepsilon})_{x_i})\,dx + \int_B |f_{\varepsilon}|\,|u_{\varepsilon}-U_{\varepsilon}|\,dx.
\] 
We now use the fact that 
\begin{equation}
\label{gi}
\frac{1}{p}\left( \frac{1}{2^{p-1}}|t|^p -\delta^p\right) \leq g_{i,\varepsilon}(t) \leq \frac{2}{p}|t|^p +\frac{p-2}{2p}.
\end{equation}
The lower bound in \eqref{gi} follows from 
\[
|t|^p\leq 2^{p-1}\,((|t|-\delta_i)_{+}^p+\delta_{i}^p),
\]
while the upper bound is a consequence of Young inequality. This implies
\[
\sum_{i=1}^N \int_B |(u_{\varepsilon})_{x_i}|^p \,dx \leq C\sum_{i=1}^N \int_B|(U_{\varepsilon})_{x_i}|^p + \int_B |f_{\varepsilon}||u_{\varepsilon}-U_{\varepsilon}|\,dx + C
\]
where $C=C(N,p,\delta,|B|)>0$ depends on \(N, p, \delta\) and \(|B|\) only. By using \(\|f_{\varepsilon}\|_{L^{p'}(B)}\le \|f\|_{L^{p'}(2\,B)}\) and \eqref{Uepsilon},
standard computations lead to the desired conclusion.
\end{proof}
\begin{lm}[Regularity of the minimizer I]
\label{lm:approximationII}
Let $u_\varepsilon$ still denote the unique minimizer of \eqref{approximated}. Then we have $u_\varepsilon\in L^\infty(B)$. 
\par
Moreover, if $f\in L^\infty_{loc}(\Omega)$, then there exists a constant \(M\) independent of $\varepsilon$ such that
\begin{equation}
\label{max}
\|u_\varepsilon\|_{L^\infty(B)}\le M.
\end{equation}
\end{lm}
\begin{proof}
We use again \eqref{gi}.
This implies that the integrand \eqref{integrand} satisfies
\begin{equation}
\label{growthL}
c\,|z|^p-\|f_\varepsilon\|_{L^\infty(B)}\, |u|-C'\le L_\varepsilon(x,u,z)\le \frac{1}{c}\,|z|^p+\|f_\varepsilon\|_{L^\infty(B)}\, |u|+C',
\end{equation}
with $c=c(N,p)>0$, $C=C(N,p)>0$ and $C'=C'(p,\delta)>0$. Thus $u_\varepsilon\in L^\infty(B)$ by \cite[Theorem 7.5 \& Remark 7.6]{Gi}.
\vskip.2cm
If $f\in L^\infty_{loc}(\Omega)$, we have $\|f_\varepsilon\|_{L^\infty(B)}\le \|f\|_{L^\infty(2\,B)}$. By \eqref{growthL} we thus get that $L_\varepsilon$ satisfies growth conditions independent of $\varepsilon$.
Then by using again the a priori estimate of \cite[Theorem 7.5]{Gi} and \eqref{uniformeg} we get the desired conclusion.
\end{proof}
\begin{oss}
The previous $L^\infty$ estimate uniform in $\varepsilon$ will be needed in the proof of Theorem B.
\end{oss}
The following result is not optimal, but it is suitable to our needs.
\begin{lm}[Regularity of the minimizer II]
\label{lm:approximation} 
Let $u_\varepsilon$ still denote the unique minimizer of \eqref{approximated}. 
We have $u_\varepsilon\in C^{k}_{loc}(B)$, where
\[
k=\left\{\begin{array}{rl}
2,& \mbox{ if } 2<p\le 3,\\
3,& \mbox{ if } p>3.
\end{array}
\right.
\]
\end{lm}
\begin{proof}
We divide the proof in two parts.
\vskip.2cm\noindent
{\it Local Lipschitz regularity.} By Lemma \ref{lm:approximationII}, we know that $u_\varepsilon$ is bounded. Then the local Lipschitz continuity is a plain consequence of Theorem \ref{teo:boucla}, applied with
\begin{equation}
\label{Fepsilon}
F_\varepsilon(z)=\sum_{i=1}^N g_{i}(z_i)+\frac{\varepsilon}{2}\, |z|^2,\qquad z\in\mathbb{R}^N,
\end{equation}
which verifies \eqref{dessous} with $\mu=\varepsilon>0$. 
\vskip.2cm\noindent
{\it Local higher regularity.} Let $\widetilde B\Subset B$ be a ball and set $\ell=\|\nabla u_\varepsilon\|_{L^\infty(\widetilde B)}$, which is finite thanks to the previous step.
By optimality, we have that $u_\varepsilon$ solves the  
elliptic equation
\begin{equation}
\label{DGNM}
\mathrm{div}(\nabla F_\varepsilon(\nabla u_\varepsilon))=f_\varepsilon,\qquad \mbox{ in } \widetilde B,
\end{equation}
where $F_\varepsilon$ is as in \eqref{Fepsilon}. Since we have
\[
\varepsilon\,|\xi|^2\le \langle D^2 F_\varepsilon(\nabla u_\varepsilon)\,\xi,\xi\rangle\le \left(\varepsilon+(p-1)\, \ell^{p-2}\right)\, |\xi|^2,\qquad \mbox{ on }\widetilde B, 
\]
we can infer $u_\varepsilon\in W^{2,2}_{loc}(\widetilde B)$ by a standard differential quotients argument (see for example \cite[Theorem 8.1]{Gi}). This in turn permits to find the equation locally solved by $\nabla u_\varepsilon$, by differentiating \eqref{DGNM}. Thus $\nabla u_\varepsilon\in C^{0,\sigma}_{loc}(\widetilde B)$ by the celebrated De Giorgi--Moser--Nash Theorem, for some $\sigma>0$. It remains to observe that $F_\varepsilon\in C^{k,\alpha}$, where 
$k$ is as in the statement and 
\[
\alpha=\left\{\begin{array}{rl}
\min\{p-2,1\},& \mbox{ if } 2<p\le 3,\\
\min\{p-3,1\},& \mbox{ if } p>3.
\end{array}
\right.
\]
Then \cite[Theorem 10.18]{Gi} implies that $u_\varepsilon$ has the claimed regularity properties.
\end{proof}
\begin{lm}[Convergence to a minimizer]
\label{lm:convergence}
With the same notation as before, we have
\[
\lim_{\varepsilon\to 0} \|u_\varepsilon-\widetilde u\|_{L^p(B)}=0,
\]
where $\widetilde u$ is a solution of
\begin{equation}
\label{local}
\min\left\{\mathfrak{F}(\varphi;B)\, :\, \varphi-U\in W^{1,p}_0(B)\right\}.
\end{equation}
\end{lm}
\begin{proof}
By \eqref{uniformeg}, there exists a sequence $\{\varepsilon_k\}_{k\in\mathbb{N}}$ converging to $0$ as $k$ goes to $\infty$ and a function $\widetilde u\in W^{1,p}(B)$ such that $\{u_{\varepsilon_k}\}_{k\in\mathbb{N}}$ converges weakly to $\widetilde u$ in $W^{1,p}(B)$ and strongly in \(L^{p}(B)\).
The function $U_\varepsilon=U\ast\varrho_\varepsilon$ is of course admissible for the approximated problem and thus
\[
\liminf_{k\to\infty}\mathfrak{F}_{\varepsilon_k}(U_{\varepsilon_k};B)\ge \liminf_{k\to \infty}\mathfrak{F}_{\varepsilon_k}(u_
{\varepsilon_k};B)\ge \liminf_{k\to \infty} \mathfrak{F}(u_{\varepsilon_k};B)\ge \mathfrak{F}(\widetilde u;B),
\]
where we used the  weak lower semicontinuity of $\mathfrak{F}$. We then observe that by using the strong convergence of $U_\varepsilon$ to $U$ and inequality \eqref{lipschitz} in Appendix \ref{sec:A}, we get
\[
\begin{split}
\lim_{k\to \infty}|\mathfrak{F}_{\varepsilon_k}(U_{\varepsilon_k};B)-\mathfrak{F}(U;B)|&\le \lim_{k\to \infty} \sum_{i=1}^N \int_B |g_i((U_{\varepsilon_k})_{x_i})-g_i(U_{x_i})|\, dx\\
&+\lim_{k\to \infty} \frac{\varepsilon_k}{2}\, \int_B |\nabla U_{\varepsilon_k}|^2+\lim_{k\to \infty} \int_B |f_{\varepsilon_k}\, U_{\varepsilon_k}-f\,U|\, dx=0,
\end{split}
\]
and thus 
\[
\mathfrak{F}(U;B)=\lim_{k\to \infty} \mathfrak{F}_{\varepsilon_k}(U_{\varepsilon_k};B)\ge \mathfrak{F}(\widetilde u;B).
\]
By definition of local minimizer, the function $U$ itself is a solution of \eqref{local}, then the previous inequality implies that $\widetilde u$ is a minimizer.
\end{proof}

\section{Local energy estimates for the approximating problem}

\label{sec:caccioppoli}

For the ball $B\Subset\Omega$ we consider the regularized problem \eqref{approximated}. We still denote by $u_\varepsilon$ its unique solution, which verifies the Euler-Lagrange equation
\begin{equation}
\label{regolareg}
\sum_{i=1}^N \int g'_{i,\varepsilon}((u_\varepsilon)_{x_i})\, \varphi_{x_i}\, dx+\int f_\varepsilon\, \varphi\, dx=0,\qquad \varphi\in W^{1,p}_0(B).
\end{equation}
From now on, in order to simplify the notation, we will systematically forget the subscript $\varepsilon$ on $u_\varepsilon$ and simply write $u$.
\par
We now insert a test function of the form $\varphi=\psi_{x_j}\in W^{1,p}_0(B)$ in \eqref{regolareg}, compactly supported in $B$. Then an integration by parts lead us to
\begin{equation}
\label{derivatag}
\sum_{i=1}^N \int g_{i,\varepsilon}''(u_{x_i})\, u_{x_i\,x_j}\, \psi_{x_i}\, dx-\int f_\varepsilon\,\psi_{x_j}\,dx=0,
\end{equation}
for $j=1,\dots,N$. This is the equation solved by $u_{x_j}$. 

\subsection{Caccioppoli-type inequalities}
In what follows we use the parameter $\delta$ defined in \eqref{delta}.
The general Caccioppoli inequality for an important class of subsolutions is given by the following result.
\begin{lm}
Let $\Phi:\mathbb{R}\to\mathbb{R}^+$ be a $C^2$ convex function 
such that
\begin{equation}
\label{astuce}
\Phi'(t)\equiv 0\quad \mbox{ for } |t|\le \delta.
\end{equation}
Then there exists a constant $C_3=C_3(p)>0$ such that for every Lipschitz function $\eta$ with compact support in $B$, we have
\begin{equation}
\label{motherg}
\begin{split}
\sum_{i=1}^N \int_{A_j} &g''_{i,\varepsilon}(u_{x_i})\,\left|\left(\Phi(u_{x_j})\right)_{x_i}\right|^2\, \eta^2\, dx\\
&\le C_3\,\sum_{i=1}^N \int_{A_j} g''_{i,\varepsilon}(u_{x_i})\,|\Phi(u_{x_j})|^2\, |\eta_{x_i}|^2\, dx\\
&+C_3\,\int_{A_j} |f_\varepsilon|^2\, \Big[\Phi'(u_{x_j})^2+\Phi''(u_{x_j})\, \Phi(u_{x_j})\Big]\,\eta^2\, dx+C_3\,\int_{A_j} \Phi(u_{x_j})^2\, |\eta_{x_j}|^2\, dx,\\
\end{split}
\end{equation}
where we set $A_j=\{x\in B\, :\, |u_{x_j}|\ge \delta\}$.
\end{lm}
\begin{proof}
In \eqref{derivatag} we take the test function\footnote{Observe that this is an admissible test function by Lemma \ref{lm:approximation}.} $\psi=\zeta\, \Phi'(u_{x_j})$,
with $\Phi:\mathbb{R}\to\mathbb{R}^+$ as in the statement and $\zeta$  nonnegative Lipschitz function with support in $B$. We thus obtain
\[
\begin{split}
\sum_{i=1}^N \int_{A_j} g''_{i,\varepsilon}(u_{x_i})\, \left(\Phi(u_{x_j})\right)_{x_i}\, \zeta_{x_i}\, dx&+\sum_{i=1}^N \int_{A_j} g''_{i,\varepsilon}(u_{x_i})\, u_{x_i\,x_j}^2\, \Phi''(u_{x_j})\,\zeta\, dx=\int_{A_j} f_\varepsilon\, \left(\zeta\,\Phi'(u_{x_j})\right)_{x_j}\, dx.
\end{split}
\] 
Finally, we test the previous equation against $\zeta=\eta^2\, \Phi(u_{x_j})$, where $\eta$ is again a Lipschitz function with support in $B$. Then we get
\[
\begin{split}
\sum_{i=1}^N \int_{A_j}& g''_{i,\varepsilon}(u_{x_i})\, \left|\left(\Phi(u_{x_j})\right)_{x_i}\right|^2\, \eta^2\, dx+\mathcal{S}(\eta)\\
&\le 2\, \sum_{i=1}^N\int_{A_j} g''_{i,\varepsilon}(u_{x_i})\, \left|\left(\Phi(u_{x_j})\right)_{x_i}\right|\,\Phi(u_{x_j})\,\eta\,|\eta_{x_i}| dx+\int_{A_j} |f_\varepsilon|\, \left|\left(\eta^2\,\Phi(u_{x_j})\,\Phi'(u_{x_j})\right)_{x_j}\right|\, dx.
\end{split}
\]
where we have introduced the {\it sponge term}
\[
\mathcal{S}(\eta)=\sum_{i=1}^N \int_{A_j} g''_{i,\varepsilon}(u_{x_i})\, u_{x_i\,x_j}^2\, \Phi''(u_{x_j})\,\Phi(u_{x_j})\,\eta^2\, dx.
\]
From the previous inequality, by Young inequality in the first term on the right-hand side
\begin{equation}
\label{conti}
\begin{split}
\sum_{i=1}^N &\int_{A_j}  g''_{i,\varepsilon}(u_{x_i})\, \left|\left(\Phi(u_{x_j})\right)_{x_i}\right|^2\, \eta^2\, dx+2\,\mathcal{S}(\eta)\\
&\le 4\, \sum_{i=1}^N\int_{A_j} g''_{i,\varepsilon}(u_{x_i})\, \Phi(u_{x_j})^2\,|\eta_{x_i}|^2\, dx+2\,\int_{A_j} |f_\varepsilon|\, \left|\left(\eta^2\,\Phi(u_{x_j})\,\Phi'(u_{x_j})\right)_{x_j}\right|\, dx.
\end{split}
\end{equation}
We now estimate the term containing $f_\varepsilon$. We first observe
\[
\begin{split}
\int_{A_j} |f_\varepsilon|\,\Big|\left(\eta^2\,\Phi(u_{x_j})\!\right.&\left.\Phi'(u_{x_j})\right)_{x_j}\Big|\, dx\le \int_{A_j} |f_\varepsilon|\, |\Phi'(u_{x_j})|\left|\left(\Phi(u_{x_j})\right)_{x_j}\right|\,\eta^2\, dx\\
&+2\int_{A_j} |f_\varepsilon|\, |\Phi'(u_{x_j})|\, \Phi(u_{x_j})\, \eta\, |\eta_{x_j}|\, dx+\int_{A_j} |f_\varepsilon|\, \left|\left(\Phi'(u_{x_j})\right)_{x_j}\right|\, \Phi(u_{x_j})\,\eta^2\, dx.
\end{split}
\]
On the set $A_j$ we have
\begin{equation}
\label{cazzata}
g''_{j,\varepsilon}(u_{x_j})\ge  (p-1).
\end{equation}
Let us consider the first term above containing $f_\varepsilon$:
\[
\begin{split}
\int_{A_j} |f_\varepsilon|\, |\Phi'(u_{x_j})|\, \left|\left(\Phi(u_{x_j})\right)_{x_j}\right|\,\eta^2\, dx&\le \frac{1}{2\,\tau}\, \int_{A_j} |f_\varepsilon|^2\, |\Phi'(u_{x_j})|^2\, \eta^2\, dx\\
&+\frac{\tau}{2}\,\int_{A_j} \left|\left(\Phi(u_{x_j})\right)_{x_j}\right|^2\, \eta^2\,dx\\
&\le \frac{1}{2\,\tau}\, \int_{A_j} |f_\varepsilon|^2\, |\Phi'(u_{x_j})|^2\, \eta^2\, dx\\
&+\frac{\tau}{2\,(p-1)}\,\int_{A_j} g''_{j,\varepsilon}(u_{x_j})\,\left|\left(\Phi(u_{x_j})\right)_{x_j}\right|^2\, \eta^2\,dx.
\end{split}
\]
The last term can be absorbed in the left-hand side of \eqref{conti}, by taking $\tau=(p-1)/2$. The second term containing $f_\varepsilon$ is simply estimated by Young inequality
\[
\int_{A_j} |f_\varepsilon|\, |\Phi'(u_{x_j})|\, \Phi(u_{x_j})\, \eta\, |\eta_{x_j}|\, dx\le \frac{1}{2}\, \int_{A_j} |f_\varepsilon|^2\, |\Phi'(u_{x_j})|^2\, \eta^2\, dx+\frac{1}{2} \int_{A_j} \Phi(u_{x_j})^2\, |\eta_{x_j}|^2\, dx,
\]
while for the last one we use the sponge term $\mathcal{S}(\eta)$ to absorb the Hessian of u. Namely, we have
\[
\begin{split}
\int_{A_j} |f_\varepsilon|\, \left|\left(\Phi'(u_{x_j})\right)_{x_j}\right|\, \Phi(u_{x_j})\,\eta^2\, dx&=\int_{A_j} |f_\varepsilon|\, |u_{x_j\,x_j}|\, \Phi''(u_{x_j})\, \Phi(u_{x_j})\,\eta^2\, dx\\
& \le \tau\,\int_{A_j} u_{x_j\,x_j}^2\, \Phi''(u_{x_j})\, \Phi(u_{x_j})\,\eta^2\, dx\\
&+\frac{1}{\tau}\,\int_{A_j}|f_\varepsilon|^2\,\Phi''(u_{x_j})\, \Phi(u_{x_j})\,\eta^2\, dx\\
&\le \frac{\tau}{p-1}\, \mathcal{S}(\eta)+\frac{1}{\tau}\,\int_{A_j}|f_\varepsilon|^2\,\Phi''(u_{x_j})\, \Phi(u_{x_j})\,\eta^2\, dx.
\end{split}
\]
In the last estimate we used again \eqref{cazzata}.
The term $\tau/(p-1)\,\mathcal{S}(\eta)$ can then be absorbed in the left-hand side. This concludes the proof.
\end{proof}
If we allow for derivatives of $f_\varepsilon$ on the right-hand side of \eqref{motherg}, the previous estimate is simpler to get. In this case we can allow for more general subsolutions.
\begin{lm}[Right-hand side in a Sobolev space]
Let $\Phi:\mathbb{R}\to\mathbb{R}$ be a $C^1$ convex function. 
Then there exists a constant $C_3=C_3(p)>0$ such that for every Lipschitz function $\eta$ with compact support in $B$, we have
\begin{equation}
\label{mothergsob}
\begin{split}
\sum_{i=1}^N &\int g''_{i,\varepsilon}(u_{x_i})\,\left|\left(\Phi(u_{x_j})\right)_{x_i}\right|^2\, \eta^2\, dx\\
&\le C_3\,\sum_{i=1}^N \int g''_{i,\varepsilon}(u_{x_i})\,|\Phi(u_{x_j})|^2\, |\eta_{x_i}|^2\, dx+C_3\,\int |(f_\varepsilon)_{x_j}|\, |\Phi'(u_{x_j})|\, |\Phi(u_{x_j})|\,\eta^2\, dx.\\
\end{split}
\end{equation}
\end{lm}
\begin{proof}
Let us suppose for simplicity that $\Phi\in C^2$. If this were not the case, a standard smoothing argument will be needed, we leave the details to the reader.
\par
We start observing that equation \eqref{derivatag} can also be written as
\begin{equation}
\label{derivatagsob}
\begin{split}
\sum_{i=1}^N \int g''_{i,\varepsilon}(u_{x_i})\, u_{x_i\,x_j}\, \psi_{x_i}\, dx+\int (f_\varepsilon)_{x_j}\,\psi\,dx=0,\quad j=1,\dots,N.
\end{split}
\end{equation}
Then we take in \eqref{derivatagsob} the test function $\psi=\zeta\, \Phi'(u_{x_j})$ as before,
with $\Phi$ as in the statement and $\zeta$ a nonnegative Lipschitz function supported in $B$. We obtain
\[
\begin{split}
\sum_{i=1}^N \int g''_{i,\varepsilon}(u_{x_i})\, \left(\Phi(u_{x_j})\right)_{x_i}\, \zeta_{x_i}\, dx\le -\int (f_\varepsilon)_{x_j}\, \Phi'(u_{x_j})\, \zeta\, dx.
\end{split}
\] 
Finally, we take again $\zeta=\eta^2\, \Phi(u_{x_j})$,
to get
\[
\begin{split}
\sum_{i=1}^N \int g''_{i,\varepsilon}(u_{x_i})\, \left|\left(\Phi(u_{x_j})\right)_{x_i}\right|^2\, \eta^2\, dx&\le 2\, \sum_{i=1}^N\int g''_{i,\varepsilon}(u_{x_i})\, \left|\left(\Phi(u_{x_j})\right)_{x_i}\right|\,\Phi(u_{x_j})\,\eta\,|\eta_{x_i}|\, dx\\
&+ \int \left|(f_\varepsilon)_{x_j}\right|\, |\Phi'(u_{x_j})|\, |\Phi(u_{x_j})|\,\eta^2\, dx.\\
\end{split}
\]
By using Young inequality as before, we get
\[
\begin{split}
\sum_{i=1}^N \int g''_{i,\varepsilon}(u_{x_i})\, \left|\left(\Phi(u_{x_j})\right)_{x_i}\right|^2\, \eta^2\, dx&\le 4\, \sum_{i=1}^N\int g''_{i,\varepsilon}(u_{x_i})\, \Phi(u_{x_j})^2\,|\eta_{x_i}|^2\, dx\\
&+2\, \int |(f_\varepsilon)_{x_j}|\, |\Phi'(u_{x_j})|\, |\Phi(u_{x_j})|\,\eta^2\, dx.
\end{split}
\]
This concludes the proof.
\end{proof}

\subsection{A Sobolev estimate}
In what follows we set
\[
W_j=\delta^2+(|u_{x_j}|-\delta)^2_+.
\]
\begin{lm}
There exists a constant $C_4=C_4(p,\delta)>0$ such that for every Lipschitz function $\eta$ with compact support in $B$, we have
\begin{equation}
\label{sobog}
\begin{split}
\sum_{i=1}^N \int \left|\nabla W_i^\frac{p}{4}\right|^2\, \eta^2\, dx
&\le C_4\sum_{i,j=1}^N \int W_i^\frac{p-2}{2}\, W_j\,|\eta_{x_i}|^2\,dx+C_4\sum_{j=1}^N\int |(f_\varepsilon)_{x_j}|\, \sqrt{W_j}\, \eta^2\, dx.
\end{split}
\end{equation}
\end{lm}
\begin{proof}
We choose the function $\Phi(t)=t$ in \eqref{mothergsob}. First observe that
\[
\left|\left(\Phi(u_{x_j})\right)_{x_i}\right|^2=u_{x_i\,x_j}^2,
\]
and that by Lemma \ref{lm:second_dessous} for $|t|\ge \delta$
\begin{equation}
\label{basic}
\begin{split}
g''_{i,\varepsilon}(t)&\ge (p-1)\,\left(\frac{\delta-\delta_i}{\delta}\right)^{p-2}\,\left(\delta^2+(|t|-\delta)_+^2\right)^\frac{p-2}{2}\\
&\ge  \frac{1}{\delta^{p-2}}\,\left(\delta^2+(|t|-\delta)_+^2\right)^\frac{p-4}{2}\,(|t|-\delta)_+^2.
\end{split}
\end{equation}
In the second inequality above we also used that $p\ge 2$ and $\delta-\delta_i\ge 1$.
Then, by \eqref{basic} we have\footnote{Observe that the inequality holds true everywhere, not only on $A_i$, since $W_i$ is constant outside $A_i$.}
\[
\begin{split}
g''_{i,\varepsilon}(u_{x_i})\,\left|\left(\Phi(u_{x_j})\right)_{x_i}\right|^2
&\ge \frac{1}{\delta^{p-2}}\,W_i^\frac{p-4}{2}\,(|u_{x_i}|-\delta)_+^2\, u_{x_i\,x_j}^2=c\, \left|\partial_{x_j} W_i^\frac{p}{4}\right|^2,
\end{split}
\]
where $c=c(\delta,p)>0$. We further observe that
\[
|\Phi(u_{x_j})|=|u_{x_j}|\le \sqrt{2}\,\sqrt{W_j},
\]
and\footnote{We use that $t-\delta_i=(t-\delta)+(\delta-\delta_i)\le (t-\delta)_++\delta$, which implies
\[
(t-\delta_i)_+\le(t-\delta)_++\delta.
\]}
\begin{equation}
\label{kiven}
g''_{i,\varepsilon}(u_{x_i})= (p-1)\, (|u_{x_i}|-\delta_{i})_+^{p-2}+\varepsilon\le c\,W_i^\frac{p-2}{2}, 
\end{equation}
where $c=c(p)>0$. Then we get the desired result by summing \eqref{mothergsob} over $j=1,\dots,N$.
\end{proof}
In what follows, we will use for simplicity the notation
\[
\fint_E \varphi\, dx:=\frac{1}{|E|}\,\int_E\varphi\, dx.
\]
\begin{coro}
\label{coro:sobolev_coro}
There exists a constant $C'_4=C'_4(p,\delta,N)>0$ such that for every pair of concentric balls $B_{R_0}\Subset B_{\rho_0}\Subset B$, we have
\begin{equation}
\label{sobolev_coro}
\begin{split}
\sum_{j=1}^N \frac{1}{R_0^{N-2}}\,\int_{B_{R_0}} \left|\nabla W_j^\frac{p}{4}\right|^2\, dx&\le C\,\left(\frac{\rho_0}{R_0}\,\right)^{N-2}\,\left(\frac{\rho_0}{\rho_0-R_0}\right)^2\,\sum_{j=1}^N\fint_{B_{\rho_0}} W^\frac{p}{2}_j\,dx\\
&+C\,\rho_0^{\frac{2}{p-1}-N+2}\, \sum_{j=1}^N \int_{B_{\rho_0}} |(f_\varepsilon)_{x_j}|^{p'}\, dx.
\end{split}
\end{equation}
\end{coro}
\begin{proof}
It is sufficient to insert the test function
\[
\eta(x)=\min\left\{1,\frac{(\rho_0-|x|)_+}{\rho_0-R_0}\right\},
\]
in \eqref{sobog} and then use H\"older and Young inequalities in the right-hand side. These give
\[
\sum_{i,j=1}^N \int W_i^\frac{p-2}{2}\, W_j\,|\eta_{x_i}|^2\,dx\le \frac{1}{(\rho_0-R_0)^2}\,\sum_{i,j=1}^N\,\left(\int_{B_{\rho_0}} W_i^\frac{p}{2}\,dx\right)^\frac{p-2}{p}\,\left(\int_{B_{\rho_0}} W_j^\frac{p}{2}\,dx\right)^\frac{2}{p},
\]
and
\[
\sum_{j=1}^N\int |(f_\varepsilon)_{x_j}|\, \sqrt{W_j}\, \eta^2\, dx\le C\,\sum_{j=1}^N \rho_0^\frac{2}{p-1}\, \int_{B_{\rho_0}} |(f_\varepsilon)_{x_j}|^{p'}\, dx+\frac{C}{\rho_0^2}\,\sum_{j=1}^N \int_{B_{\rho_0}} W_j^\frac{p}{2}\, dx,
\]
which concludes the proof.
\end{proof}
\begin{oss}[Uniform Sobolev estimate]
\label{oss:sobolev}
From the previous result, we obtain that if $f\in W^{1,p'}_{loc}(\Omega)$, then for every $i=1,\dots,N$ the function $W_i^{p/4}$ enjoys a $W^{1,2}_{loc}(B)$ estimate independent of $\varepsilon$, thanks to \eqref{uniformeg} and 
\[
\|f_\varepsilon\|_{W^{1,p'}(B_{\rho_0})}\le \|f\|_{W^{1,p'}(2\,B)}.
\]
\end{oss}
\subsection{Power-type subsolutions}
We still use the notation 
\[
W_j=\delta^2+(|u_{x_j}|-\delta)^2_+.
\]
Then we have the following result.
\begin{lm}
\label{lm:caccioweight}
There exists a constant $C_5=C_5(p)>0$ such that for every $s\ge 0$ and every Lipschitz function $\eta$ with compact support in $B$, we have
\begin{equation}
\label{baseg}
\begin{split}
\sum_{i=1}^N \int g''_{i,\varepsilon}(u_{x_i})\,\left|\left( W_j^\frac{s+1}{2}\right)_{x_i}\right|^2\, \eta^2\, dx
&\le C_5\,\sum_{i=1}^N \int W_i^\frac{p-2}{2}\, W_j^{s+1}\,|\nabla\eta|^2\,dx\\
&+C_5\,(s+1)^2\, \int |f_\varepsilon|^2\, W_j^{s}\, \eta^2\, dx,\qquad j=1,\dots,N.
\end{split}
\end{equation}
\end{lm}
\begin{proof}
In equation \eqref{motherg} we make the choice\footnote{Observe that this function is not $C^2$, but only $C^{1,1}$ near $t=\delta$ ot $t=-\delta$. This is not a big issue, since in any case $\Phi''$ stays bounded as $|t|\to \delta$, thus we can use \eqref{motherg} for a regularization of $\Phi$ and then pass to the limit at the end.}
\[
\Phi(t)=\Big(\delta^2+(|t|-\delta)^2_+\Big)^\frac{s+1}{2},
\]
for $s\ge 0$ which satisfies hypothesis \eqref{astuce}.
Observe that by definition we have
\[
\Phi(u_{x_j})=W_j^\frac{s+1}{2},
\]
so that
\[
\left|\left(\Phi(u_{x_j})\right)_{x_i}\right|^2=\left|\left( W_j^\frac{s+1}{2}\right)_{x_i}\right|^2.
\]
Thus the left-hand side of \eqref{motherg} coincides with
\[
\begin{split}
\sum_{i=1}^N \int g_{i,\varepsilon}''(u_{x_i})\, \left|\left( W_j^\frac{s+1}{2}\right)_{x_i}\right|^2\, \eta^2\,dx.
\end{split}
\]
We now come to the right-hand side: 
\[
\begin{split}
\sum_{i=1}^N \int g''_{i,\varepsilon}(u_{x_i})\,|\Phi(u_{x_j})|^2\, |\eta_{x_i}|^2\, dx&=\sum_{i=1}^N \int g''_{i,\varepsilon}(u_{x_i})\, W_j^{s+1}\,|\eta_{x_i}|^2\, dx\\
&\le C\, \sum_{i=1}^N \int W_i^\frac{p-2}{2}\, W_j^{s+1}\,|\eta_{x_i}|^2\, dx,
\end{split}
\]
thanks to \eqref{kiven}.
For the other two terms, by using the definition of $\Phi$ we simply have
\[
\begin{split}
\int_{A_j} |f_\varepsilon|^2\, \Big[\Phi'(u_{x_j})^2+\Phi''(u_{x_j})\, \Phi(u_{x_j})\Big]\,\eta^2\, dx&+\int_{A_j} \Phi(u_{x_j})^2\, |\eta_{x_j}|^2\, dx\\
&\le C\,(s+1)^2\,\int |f_\varepsilon|^2\, W_j^{s}\, \eta^2\, dx\\
&+\sum_{i=1}^N \int W_i^\frac{p-2}{2}\,W_j^{s+1}\, |\eta_{x_j}|^2\, dx,
\end{split}
\]
where we used that
\[
\Big[\Phi'(t)^2+\Phi''(t)\,\Phi(t)\Big]\le C\,(s+1)^2\, \Big(\delta^2+(|t|-\delta)^2_+\Big)^{s}=C\,(1+s)^2\, \Phi(t)^\frac{2\,s}{s+1}, 
\]
and $W_j^{s+1}\le \sum_{i=1}^N\, W_i^\frac{p-2}{2}\, W_j^{s+1}$, which follows from $W_i\ge 1$.
\end{proof}
In particular, we get an estimate for the {\it diagonal terms}, corresponding to $i=j$.
\begin{coro}
There exists a constant $C_6=C_6(p,\delta)>0$ such that for every $s\ge 0$ and every Lipschitz function $\eta$ with compact support in $\Omega$, we have
\begin{equation}
\label{basegg}
\begin{split}
\int \left|\left(W_j^{\frac{p}{4}+\frac{s}{2}}\right)_{x_j}\right|^2\, \eta^2\, dx
&\le C_6\,\sum_{i=1}^N \int W_i^\frac{p-2}{2}\, W_j^{s+1}\,|\nabla\eta|^2\,dx\\
&+C_6\,(s+1)^2\, \int |f_\varepsilon|^2\, W_j^{s}\, \eta^2\, dx,\qquad j=1,\dots,N.
\end{split}
\end{equation}
\end{coro}
\begin{proof}
We fix $j$, by keeping only the term $i=j$ and dropping all the others in the left-hand side of \eqref{baseg}, we get
\[
\begin{split}
\int_{A_j} g_{j,\varepsilon}''(u_{x_j})\,\left|\left(W_j^\frac{s+1}{2}\right)_{x_j}\right|^2\, \eta^2\, dx&\le C_5\, \sum_{i=1}^N \int W_i^\frac{p-2}{2}\,W_j^{s+1}\,|\nabla\eta|^2\, dx\\
&+C_5\,(s+1)^2\, \int |f_\varepsilon|^2\, W_j^{s}\, \eta^2\, dx.
\end{split}
\]
We now observe that again by Lemma \ref{lm:second_dessous} on $A_j$ we have
\[
\begin{split}
g_{j,\varepsilon}''(u_{x_j})
&\ge \frac{p-1}{\delta^{p-2}}\,\left[\delta^2+(|u_{x_j}|-\delta)^2_+\right]^\frac{p-2}{2}=\frac{p-1}{\delta^{p-2}}\,W_j^\frac{p-2}{2},
\end{split}
\] 
and that
\[
W_j^\frac{p-2}{2}\,\left|\left(W_j^\frac{s+1}{2}\right)_{x_j}\right|^2=\left(\frac{2+2\,s}{p+2\,s}\right)^2\,\left|\left(W_j^{\frac{p}{4}+\frac{s}{2}}\right)_{x_j}\right|^2\ge \left(\frac{2}{p}\right)^2\, \left|\left(W_j^{\frac{p}{4}+\frac{s}{2}}\right)_{x_j}\right|^2,
\]
so that the conclusion follows.
\end{proof}

\section{Proof of Theorem A}
\label{sec:2d}

The core of the proof of Theorem A is the a priori estimate of Proposition \ref{lm:2d} below. We postpone its proof and proceed with that of Theorem A.
\begin{proof}
Let $\Omega'\Subset\Omega$ and set $d=\mathrm{dist}(\Omega',\partial \Omega)$. We take $r_0\le d/10$, then $\Omega'$ can be covered by a finite number of balls centered at points in $\Omega'$ and having radius $r_0$.
Let $B_{r_0}:=B_{r_0}(x_0)\Subset\Omega$ be one of these balls, it is clearly sufficient to show that
\[
\|\nabla U\|_{L^\infty(B_{r_0})}<+\infty.
\]
To this aim we take the solution $u_\varepsilon$ of the regularized problem \eqref{approximated} in the ball $B:=B_{4\,r_0}(x_0)$. Observe that by construction we have $2\,B=B_{8\,r_0}(x_0)\Subset\Omega$. Then there exists $\varepsilon_0=\varepsilon_0(d)>0$ such that for every $0<\varepsilon\le \varepsilon_0$
\[
\|f_\varepsilon\|_{W^{1,p'}(B)}\le \|f\|_{W^{1,p'}(2\,B)}.
\]
By using estimate \eqref{componentig} below with $R_0=2\,r_0$ and $\rho_0=3\,r_0$ we get
\begin{equation}
\label{quasi0}
\|\nabla u_\varepsilon\|_{L^\infty(B_{r_0})}\le C,\quad \mbox{ for every } 0<\varepsilon\le\varepsilon_0,
\end{equation}
where $C>0$ depends only on $p$, $\delta$, $r_0$, $\|f\|_{W^{1,p'}(2\,B)}$ and the constant $C_2$ in \eqref{uniformeg}. We then observe that by Lemma \ref{lm:convergence}, we can find a sequence $\{\varepsilon_k\}_{k\in\mathbb{N}}$ converging to $0$ and such that $\{u_{\varepsilon_k}\}$ converges strongly in $L^p(B)$ and weakly in $W^{1,p}(B)$ to a solution $\widetilde u$ of 
\[
\min\{\mathfrak{F}(\varphi;B)\, :\, \varphi-U\in W^{1,p}(B)\}.
\]
By lower semicontinuity we have that $\widetilde u$ still satisfies \eqref{quasi0}. It is now sufficient to use Lemma \ref{propagationofregularity} in order to transfer this Lipschitz estimate from $\widetilde u$ to the original local minimizer $U$. This concludes the proof.
\end{proof}
\begin{prop}[Uniform Lipschitz estimate, $N=2$]
\label{lm:2d}
Let $N=2$ and $p\ge 2$. 
Then for every triple of concentric balls $B_{r_0}\Subset B_{R_0}\Subset B_{\rho_0}\Subset B$ and $i=1,2$ we have
\begin{equation}
\label{componentig}
\|(u_\varepsilon)_{x_i}\|_{L^\infty(B_{r_0})}\le C_7\,\left(\frac{R_0}{R_0-r_0}\right)^4\,\mathcal{J}(u_\varepsilon,f_\varepsilon;R_0,\rho_0)^2\,\left[\left(\fint_{B_{R_0}} \left|(u_\varepsilon)_{x_i}\right|^p\, dx\right)^\frac{1}{p}+\delta\right],
\end{equation}
where $C_7=C_7(p,\,\delta)>0$ is a constant that only depends on $p$ and $\delta$ and
\begin{equation}
\label{J}
\begin{split}
\mathcal{J}(u_\varepsilon,f_\varepsilon;R_0,\rho_0)&=\left(\frac{\rho_0}{\rho_0-R_0}\right)^2\,\left[\fint_{B_{\rho_0}} |\nabla u_\varepsilon|^p\,dx+\delta^p\right]\\
&+\rho_0^{\frac{2}{p-1}}\, \int_{B_{\rho_0}} |\nabla f_\varepsilon|^{p'}\, dx+\rho_0^{\frac{2}{p}}\, \left(\int_{B_{\rho_0}} |f_\varepsilon|^{2\,p'}\, dx\right)^\frac{1}{p'}.
\end{split}
\end{equation}
\end{prop}
\begin{proof}
For notational simplicity, we write again $u$ in place of $u_\varepsilon$. We still use the notation
\[
W_j=\delta^2+\left(|u_{x_j}|-\delta\right)^2_+,\qquad j=1,2.
\]
We give the proof for $u_{x_1}$, the one for $u_{x_2}$ being exactly the same. By \eqref{basegg} we already know that
\[
\begin{split}
\int \left|\left(W_1^{\frac{p}{4}+\frac{s}{2}}\right)_{x_1}\right|^2\, \eta^2\, dx
&\le C_6\,\sum_{i=1}^2\int W_i^\frac{p-2}{2}\, W_1^{s+1}\,|\nabla\eta|^2\,dx+C_6\,(s+1)^2 \int |f_\varepsilon|^2\, W_1^{s}\, \eta^2\, dx,
\end{split}
\]
where $\eta$ is any Lipschitz function supported on $B$ and such that $0\le \eta \le 1$. We add the term
\[
\int \left|\eta_{x_1}\right|^2\, W_1^{\frac{p}{2}+s}\, dx,
\]
on both sides of the previous inequality
and observe that
\[
\int \left|\left(W_1^{\frac{p}{4}+\frac{s}{2}}\right)_{x_1}\right|^2\, \eta^2\, dx+\int  W_1^{\frac{p}{2}+s}\,|\eta_{x_1}|^2\, dx\ge \frac{1}{2}\,\int \left|\left(W_1^{\frac{p}{4}+\frac{s}{2}}\,\eta\right)_{x_1}\right|^2\, dx.
\]
We thus obtain
\begin{equation}
\label{x1}
\begin{split}
\int \left|\left(W_1^{\frac{p}{4}+\frac{s}{2}}\,\eta\right)_{x_1}\right|^2\, dx&\le C\,\sum_{i=1}^2\int W_i^\frac{p-2}{2}\, W_1^{s+1}\,|\nabla\eta|^2\,dx+C\,(s+1)^2\, \int |f_\varepsilon|^2\, W_1^{s}\, \eta^2\, dx,
\end{split}
\end{equation}
with $C=C(p,\delta)>0$.
\par
The main problem of the Caccioppoli inequality \eqref{baseg} is that apparently we can not use it to control the {\it missing term}
\[
\left(W_1^{\frac{p}{4}+\frac{s}{2}}\right)_{x_2}.
\]
Thus there is an obstruction to derive estimates for $\nabla W_1^{\frac{p}{4}+\frac{s}{2}}$ which could lead to an interative scheme of reverse H\"older inequalities. In order to overcome this problem, we observe that 
\[
\left|\left(W_1^{\frac{p}{4}+\frac{s}{2}}\right)_{x_2}\right|=\frac{p+2\,s}{p}\,\left|\left( W_1^{\frac{p}{4}}\right)_{x_2}\right|\, W_1^\frac{s}{2}.
\]
Then if we fix $1<q<2$, by H\"older inequality with exponents $2/q$ and $2/(2-q)$,
we have
\[
\begin{split}
\left(\int \left|\left(W_1^{\frac{p}{4}+\frac{s}{2}}\right)_{x_2}\right|^q\, \eta^q\, dx\right)^\frac{2}{q}&\le\left(\frac{p+2\,s}{p}\right)^2\left(\int \left|\left(W_1^{\frac{p}{4}}\right)_{x_2}\right|^2\eta^2\, dx\right) \left(\int_{\mathrm{spt}(\eta)} W_1^{\frac{q}{2-q}\,s}\, dx\right)^\frac{2-q}{q}\!\!.\\
\end{split}
\]
The precise value of $q$ will be specified later.  
We now add the term
\[
\left(\int W_1^{\frac{p\,q}{4}+\frac{s\,q}{2}}\, |\eta_{x_2}|^q\, dx\right)^\frac{2}{q},
\]
on both sides of the previous inequality and observe that by triangle inequality
\[
\left(\int \left|\left(W_1^{\frac{p}{4}+\frac{s}{2}}\right)_{x_2}\right|^q\, \eta^q\, dx\right)^\frac{2}{q}+\left(\int W_1^{\frac{p\,q}{4}+\frac{s\,q}{2}}\, |\eta_{x_2}|^q\, dx\right)^\frac{2}{q}\ge \frac{1}{2}\, \left(\int \left|\left(W_1^{\frac{p}{4}+\frac{s}{2}}\,\eta\right)_{x_2}\right|^q\, dx\right)^\frac{2}{q}.
\]
Thus we get
\begin{equation}
\label{x2}
\begin{split}
\left(\int \left|\left(\left(W_1^{\frac{p}{4}+\frac{s}{2}}\right)_{x_2}\eta\right)\right|^q\, dx\right)^\frac{2}{q}&\le C\,(1+s)^2\,\left(\int \left|\left(W_1^{\frac{p}{4}}\right)_{x_2}\right|^2\eta^2\, dx\right)\,\left(\int_{\mathrm{spt}(\eta)} W_1^{\frac{q}{2-q}\,s}\, dx\right)^\frac{2-q}{q}\\
&+C\,\left(\int W_1^{\frac{p\,q}{4}+\frac{s\,q}{2}}\, |\eta_{x_2}|^q\, dx\right)^\frac{2}{q},
\end{split}
\end{equation}
with $C=C(p)>0$. For $0<r<R<R_0$, we now take $\eta\in W^{1,\infty}_0(B_R)$ to be the standard cut-off function
\[
\eta(x)=\min\left\{1,\frac{(R-|x|)_+}{R-r}\right\},
\]
then by multiplying \eqref{x1} and \eqref{x2} we get
\begin{equation}
\label{bordel!}
\begin{split}
\left(\int \left|\left(W_1^{\frac{p}{4}+\frac{s}{2}}\,\eta\right)_{x_1}\right|^2\, dx\right)&\left(\int \left|\left(W_1^{\frac{p}{4}+\frac{s}{2}}\,\eta\right)_{x_2}\right|^q\, dx\right)^\frac{2}{q}\\
&\le C\,\left[\frac{1}{(R-r)^2}\,\sum_{i=1}^2 \int_{B_R} W_i^\frac{p-2}{2}\, W_1^{s+1}\, dx+(s+1)^2\, \int_{B_R} |f_\varepsilon|^2\, W_1^{s}\, dx\right]\\
&\times \left[(s+1)^2\,\left(\int_{B_R} \left|\left(W_1^{\frac{p}{4}}\right)_{x_2}\right|^2\, dx\right)\,\left(\int_{B_R} W_1^{\frac{q}{2-q}\,s}\, dx\right)^\frac{2-q}{q}\right.\\
&\left.+\frac{1}{(R-r)^2}\,\left(\int_{B_R} W_1^{\frac{p\,q}{4}+\frac{s\,q}{2}}\, dx\right)^\frac{2}{q}\right].
\end{split}
\end{equation}
We now estimate the terms appearing in \eqref{bordel!}. To this aim, it will be useful to introduce the quantity
\begin{equation}
\label{mathcalI}
\begin{split}
\mathcal{I}(W_1,W_2,f_\varepsilon;R_0)&=\sum_{i=1}^2 \left[\fint_{B_{R_0}} W_i^\frac{p}{2}\, dx +\int_{B_{R_0}} \left|\nabla W_i^\frac{p}{4}\right|^2\, dx\right]\\
&+R_0^{\frac{2}{p}}\, \left(\int_{B_{R_0}} |f_\varepsilon|^{2\,p'}\, dx\right)^\frac{1}{p'}.
\end{split}
\end{equation}
Then we start with the first term on the right-hand side of \eqref{bordel!}. Observe that 
\[
\sum_{i=1}^2 \int_{B_R} W_i^\frac{p-2}{2}\, W_1^{s+1}\, dx=\int_{B_R} W_1^\frac{p}{2}\, W_1^{s}\, dx+\int_{B_R} W_2^\frac{p-2}{2}\, W_1\,W_1^{s}\, dx.
\]
We use H\"older inequality 
in conjunction with Sobolev-Poincar\'e inequality\footnote{Since we are in dimension $N=2$, we have $W^{1,2}(B_{R_0})\hookrightarrow L^{2\,p'}(B_{R_0})$. Then we have
\[
\left(\int_{B_R} \left(W^\frac{p}{4}\right)^{2\,p'}\, dx\right)^\frac{1}{p'}\, dx\le C\,R^{\frac{2}{p'}}\,\left[\fint_{B_R} \left(W^\frac{p}{4}\right)^2\, dx +\int_{B_R} \left|\nabla W^\frac{p}{4}\right|^2\, dx\right],
\]
with a constant $C=C(p)>0$.}, to get
\begin{equation}
\label{firstW1}
\begin{split}
\int_{B_R} W_1^\frac{p}{2}\,W_1^{s}\, dx&\le C\,\left[\fint_{B_{R_0}} W_1^\frac{p}{2}\, dx +\int_{B_{R_0}} \left|\nabla W_1^\frac{p}{4}\right|^2\, dx\right]\,R_0^{\frac{2}{p'}} \left(\int_{B_R} W_1^{s\,p}\, dx\right)^\frac{1}{p}\\
&\le C\, \mathcal{I}(W_1,W_2,f_\varepsilon;R_0)\, R_0^{\frac{2}{p'}} \left(\int_{B_R} W_1^{s\,p}\, dx\right)^\frac{1}{p},
\end{split}
\end{equation}
and
\begin{equation}
\label{secondW1}
\begin{split}
\int_{B_R} W_2^\frac{p-2}{2}\,W_1\,W_1^{s}\, dx&\le C\,\left[\sum_{i=1}^2\left(\int_{B_{R_0}} \left(W_i^\frac{p}{4}\right)^{2\,p'}\, dx\right)^\frac{1}{p'}\right]\,\left(\int_{B_R} W_1^{s\,p}\, dx\right)^\frac{1}{p}\\
&\le C\,\left\{\sum_{i=1}^2\,\left[\fint_{B_{R_0}} W_i^\frac{p}{2}\, dx +\int_{B_{R_0}} \left|\nabla W^\frac{p}{4}_i\right|^2\, dx\right]\right\}\,R_0^{\frac{2}{p'}} \left(\int_{B_R} W_1^{s\,p}\, dx\right)^\frac{1}{p}\\
&\le C\, \mathcal{I}(W_1,W_2,f_\varepsilon;R_0)\, R_0^{\frac{2}{p'}} \left(\int_{B_R} W_1^{s\,p}\, dx\right)^\frac{1}{p},
\end{split}
\end{equation}
for some constant $C=C(p)>0$ depening only on $p$.
\vskip.2cm\noindent
The term containing $f_\varepsilon$ in \eqref{bordel!} is estimated as follows. Observe that\footnote{The exponent $2\,p'$ is well-defined even in the case $p=2$.}
\[
W^{1,p'}(B_{R_0})\hookrightarrow L^{2\,p'}(B_{R_0}),\qquad \mbox{ since } 2\,p'<(p')^*,
\] 
then by H\"older's inequality and the definition of $\mathcal{I}(W_1,W_2,f_\varepsilon;R_0)$
\begin{equation}
\label{assorbire}
\begin{split}
\int_{B_R} |f_\varepsilon|^2\, W_1^{s}\,dx&\le \left(\int_{B_{R_0}} |f_\varepsilon|^{2\,p'}\, dx\right)^\frac{1}{p'}\,\left(\int_{B_{R}} W_1^{s\,p}\, dx\right)^\frac{1}{p}\\
&\le \mathcal{I}(W_1,W_2,f_\varepsilon;R_0)\, R_0^{-\frac{2}{p}}\,\left(\int_{B_{R}} W_1^{s\,p}\, dx\right)^\frac{1}{p}
\end{split}
\end{equation}
For the last term on the right-hand side of \eqref{bordel!}, by H\"older inequality and estimate \eqref{firstW1} we have
\begin{equation}
\label{pezzettino}
\begin{split}
\left(\int_{B_R} W_1^{\frac{p\,q}{4}+\frac{s\,q}{2}}\, dx\right)^\frac{2}{q}&\le C\,R_0^{2\,\left(\frac{2}{q}-1\right)}\int_{B_R} W_1^\frac{p}{2}\,W_1^s\, dx\\
&\le C\,R_0^{2\,\left(\frac{2}{q}-\frac{1}{p}\right)}\,\mathcal{I}(W_1,W_2,f_\varepsilon;R_0)\, \left(\int_{B_R} W_1^{s\,p}\, dx\right)^\frac{1}{p},
\end{split}
\end{equation}
where $C=C(q)>0$.
\vskip.2cm\noindent
Finally, for the left-hand side of \eqref{bordel!}, we have
\begin{equation}
\label{gauche}
\begin{split}
\left(\int \left|\left(W_1^{\frac{p}{4}+\frac{s}{2}}\,\eta\right)_{x_1}\right|^2\, dx\right)& \left(\int \left|\left(W_1^{\frac{p}{4}+\frac{s}{2}}\,\eta\right)_{x_2}\right|^q\, dx\right)^\frac{2}{q}\ge \mathcal{T}^2_{q}\, \left(\int \left(W_1^{\frac{p}{4}+\frac{s}{2}}\,\eta\right)^{\overline q^*}\,dx\right)^\frac{4}{\overline q^*}.
\end{split}
\end{equation}
Here we used the {\it anisotropic Sobolev-Troisi inequality} (see Appendix \ref{sec:B}) for the compactly supported function $W_1^{(p+2\,s)/4}\,\eta$. The exponent $\overline q^*$ is defined by
\[
\overline q^*=\frac{2\,\overline q}{2-\overline q},\qquad \mbox{ where }\quad \frac{1}{\overline q}=\frac{1}{2}\, \left(\frac{1}{2}+\frac{1}{q}\right),
\] 
so that
\[
\overline{q}=\frac{4\,q}{2+q}\qquad \mbox{ and }\qquad \overline q^*=\frac{4\,q}{2-q},
\]
the constant $\mathcal{T}_q$ only depends on $q$ and it converges to $0$ as $q$ goes to $2$.
\vskip.2cm\noindent
By using \eqref{firstW1}, \eqref{secondW1}, \eqref{assorbire}, \eqref{pezzettino} and \eqref{gauche} in \eqref{bordel!}, we then arrive at
\begin{equation}
\label{reverseg}
\begin{split}
\left(\int_{B_r} \left(W_1^{\frac{p}{2}+s}\right)^\frac{2\,q}{2-q}\, dx\right)^\frac{2-q}{q}&\le C\,\left[\left(\frac{R_0}{R-r}\right)^2\,\mathcal{I}(W_1,W_2,f_\varepsilon;R_0)\,R_0^{-\frac{2}{p}} \left(\int_{B_R} W_1^{s\,p}\, dx\right)^\frac{1}{p}\right.\\
&\left.+(s+1)^2\,\mathcal{I}(W_1,W_2,f_\varepsilon;R_0)\,R_0^{-\frac{2}{p}}\,\left(\int_{B_{R}} W_1^{s\,p}\, dx\right)^\frac{1}{p}\right]\\
&\times \left[(s+1)^2\,\mathcal{I}(W_1,W_2,f_\varepsilon;R_0)\left(\int_{B_R} W_1^{\frac{q}{2-q}\,s}\, dx\right)^\frac{2-q}{q}\right.\\
&\left.+\left(\frac{R_0}{R-r}\right)^2\,R_0^{2\,\left(\frac{2}{q}-\frac{1}{p}-1\right)}\,\mathcal{I}(W_1,W_2,f_\varepsilon;R_0)\, \left(\int_{B_R} W_1^{s\,p}\, dx\right)^\frac{1}{p}\right],
\end{split}
\end{equation}
for a constant $C=C(p,q,\delta)>0$. We now choose $1<q<2$ as follows
\begin{equation}
\label{choose!}
q= \frac{2\,p}{p+1}.
\end{equation}
Observe that with such a choice, we have
\[
\frac{q}{2-q}=p\qquad \mbox{ and }\qquad \frac{2}{q}-\frac{1}{p}-1=0.
\]
We further observe that
\[
\left(W_1^{\frac{p}{2}+s}\right)^{2\,p}\ge W_1^{2\,s\,p}, 
\]
since $W_i\ge 1$. Then \eqref{reverseg} becomes
\[
\begin{split}
\left(\int_{B_r} W_1^{2\,p\,s}\, dx\right)^\frac{1}{p}&\le C\,\mathcal{I}(W_1,W_2,f_\varepsilon;R_0)^2\left[\left(\frac{R_0}{R-r}\right)^2+(s+1)^2\,\right]^2\,R_0^{-\frac{2}{p}}\,\left(\int_{B_R} W_1^{s\,p}\, dx\right)^\frac{2}{p}.
\end{split}
\]
By using that $R_0/(R-r)\ge 1$ and $(s+1)\ge 1$ and introducing the notation $\vartheta=p\,s$, then the previous estimate finally gives
\begin{equation}
\label{reverseg3}
\begin{split}
\|W_1\|_{L^{2\,\vartheta}(B_r)}&\le \left[C\,\mathcal{I}(W_1,W_2,f_\varepsilon;R_0)\,\left(\frac{R_0}{R-r}\right)^2\, \left(\frac{\vartheta}{p}+1\right)^2\,\right]^\frac{p}{\vartheta}\,R_0^{-\frac{1}{\vartheta}}\,\|W_1\|_{L^{\vartheta}(B_R)},
\end{split}
\end{equation}
possibly for a different constant $C=C(p,\delta)>0$. {\it This is the iterative scheme of reverse H\"older inequalities needed to launch a Moser's iteration}.
\par
We then fix the two radii $R_0>r_0>0$ of the statement and consider the sequences
\[
r_k=r_0+\frac{R_0-r_0}{2^{k}}\qquad \mbox{ and }\qquad \vartheta_{k}=2\, \vartheta_{k-1}=2^{k}\, \vartheta_0=2^{k-1}\,p.
\]
Then iterating \eqref{reverseg3} infinitely many times with $R=r_k$ and $r=r_{k+1}$, we get
\[
\|W_1\|_{L^\infty(B_{r_0})}\le C\, \left(\frac{R_0}{R_0-r_0}\right)^8\,\mathcal{I}(W_1,W_2,f_\varepsilon;R_0)^4\, \left(\fint_{B_{R_0}} W_1^\frac{p}{2}\, dx\right)^\frac{2}{p},
\] 
for some constant $C=C(p,\delta)>0$.
We notice that $u_{x_1}^2\le W_1\le u_{x_1}^2+\delta^2$, by definition of $W_1$. Then we obtain with simple manipulations
\[
\|u_{x_1}\|_{L^\infty(B_{r_0})}\le C\,\left(\frac{R_0}{R_0-r_0}\right)^4\,\mathcal{I}(W_1,W_2,f_\varepsilon;R_0)^2\,\left[\left(\fint_{B_{R_0}} \left|u_{x_1}\right|^p\, dx\right)^\frac{1}{p}+\delta\right],
\]
for a possibly different constant $C=C(p,\delta)>0$.
Finally, by Corollary \ref{coro:sobolev_coro} the term $\mathcal{I}(W_1,W_2,f_\varepsilon;R_0)$ defined in \eqref{mathcalI} can be estimated as follows
\[
\begin{split}
\mathcal{I}(W_1,W_2,f_\varepsilon;R_0)&\le C\,\left(\frac{\rho_0}{\rho_0-R_0}\right)^2\,\left[\sum_{j=1}^2\fint_{B_{\rho_0}} |u_{x_j}|^p\,dx+\delta^p\right]\\
&+C\,\rho_0^{\frac{2}{p-1}}\, \int_{B_{\rho_0}} |\nabla f_\varepsilon|^{p'}\, dx+\rho_0^{\frac{2}{p}}\, \left(\int_{B_{\rho_0}} |f_\varepsilon|^{2\,p'}\, dx\right)^\frac{1}{p'}.
\end{split}
\]
This concludes the proof.
\end{proof}
\begin{oss}
\label{oss:2d?}
Observe that the previous strategy does not seem to work for $N\ge 3$. Indeed, in this case we would have $N-1$ missing terms, i.e. 
\[
\partial_{x_i} W_1^{\frac{p}{4}+\frac{s}{2}},\qquad i=2,\dots,N.
\]
By proceeding as before for each of these terms, i.e. combining \eqref{sobog} and H\"older inequality, one would have on the left-hand side the term
\[
\left(\int \left|\left(W_1^{\frac{p}{4}+\frac{s}{2}}\,\eta\right)_{x_1}\right|^2\, dx\right)\, \prod_{i=2}^N  \left(\int \left|\left(W_1^{\frac{p}{4}+\frac{s}{2}}\,\eta\right)_{x_i}\right|^q\, dx\right)^\frac{2}{q},
\]
which in turn can be estimated from below by Sobolev-Troisi inequality by
\[
\left(\int_{B_r} \left(W_1^{\frac{p}{4}+\frac{s}{2}}\right)^{\overline q^*}\,dx\right)^\frac{2}{\overline q^*}.
\]
The right-hand side would still contain the term
\[
\left(\int_{B_R} W_1^{\frac{q}{2-q}\,s}\, dx\right)^\frac{2-q}{q}.
\]
The exponent $\overline q^*$ is now defined by
\[
\overline q^*=\frac{N\,\overline q}{N-\overline q},\qquad \mbox{ where }\quad \frac{1}{\overline q}=\frac{1}{N}\, \left(\frac{1}{2}+\frac{N-1}{q}\right),
\] 
so that
\[ 
\overline{q}=\frac{2\,N\,q}{2\,N+q-2}\qquad \mbox{ and }\qquad \overline q^*=\frac{2\,N\,q}{2\,N-q-2}.
\]
Then Moser's iteration would work if
\[
\frac{\overline q^*}{2}\,s> \frac{q}{2-q}\,s\qquad\Longleftrightarrow\qquad q<\frac{2}{N-1}. 
\]
Of course, when $N\ge 3$ the last condition does not fit with the requirement $q>1$.
\end{oss}

\section{Proof of Theorem B}
\label{sec:d}

\begin{proof}
The proof is the same as that of Theorem A. The essential point is the uniform Lipschitz estimate of Proposition \ref{lm:berrrnstein} below, which replaces that of Proposition \ref{lm:2d}.
\end{proof}

\begin{prop}[Uniform Lipschitz estimate, $p\ge 4$]
\label{lm:berrrnstein}
Let $N\ge 3$ and $p\ge 4$. 
For every pair of concentric balls $B_{r_0}\Subset B_{R_0}\Subset B$, we have
\begin{equation}
\label{componentig2}
\|\nabla u_\varepsilon\|_{L^\infty(B_{r_0})}\le C_8,
\end{equation}
where $C_8=C_8(N,p,\,\delta,\,R_0-r_0,\,M,\,\|f\|_{W^{1,\infty}(2\,B)})>0$ does not depend on $\varepsilon$. Here the constant $M$ is the same appearing in \eqref{max}.
\end{prop}

\begin{proof}
As usual, for notational simplicity we simply write $u$ in place of $u_\varepsilon$.
By Lemma \ref{lm:approximation}, we get that $u$ is indeed a local $C^3$ solution of the equation \eqref{regolareg} in $B$, i.e. it verifies
\begin{equation}
\label{UU}
\sum_{i=1}^N \left(g'_{i,\varepsilon}(u_{x_i})\right)_{x_i}=f_\varepsilon,\qquad \mbox{ in } B',
\end{equation}
for every $B'\Subset B$.
This means that pointwise we have
\begin{equation}
\label{pointwise}
\sum_{i=1}^N g_{i,\varepsilon}''(u_{x_i})\, u_{x_i\,x_i}=f_\varepsilon\, \qquad \mbox{ in } B'.
\end{equation}
We now derive the previous equation with respect to $x_j$ and obtain
\begin{equation}
\label{derivala}
\sum_{i=1}^N \Big[g_{i,\varepsilon}'''(u_{x_i})\, u_{x_i\,x_i}\,u_{x_i\,x_j}+g''_{i,\varepsilon}(u_{x_i})\, u_{x_i\,x_i\,x_j}\Big]=(f_\varepsilon)_{x_j},\qquad \mbox{ in } B'.
\end{equation}
We introduce the following linear differential operator
\begin{equation}
\label{operator}
L[\psi]=\sum_{i=1}^N \Big[g_{i,\varepsilon}'''(u_{x_i})\, u_{x_i\,x_i}\,\psi_{x_i}+g''_{i,\varepsilon}(u_{x_i})\, \psi_{x_i\,x_i}\Big],
\end{equation}
then \eqref{derivala} can be simply written as $L[u_{x_j}]=(f_\varepsilon)_{x_j}$.
Also observe that
\[
L[\varphi\,\psi]=\varphi\, L[\psi]+\psi\,L[\varphi]+2\,\sum_{i=1}^N g''_{i,\varepsilon}(u_{x_i})\, \varphi_{x_i}\,\psi_{x_i}.
\]
Thus for $u_{x_j}^2$ we obtain 
\[
L[u_{x_j}^2]= 2\,u_{x_j}\, L[u_{x_j}]+2\,\sum_{i=1}^N g''_{i,\varepsilon}(u_{x_i})\, \left[\left(u_{x_j}\right)_{x_i}\right]^2=2\,u_{x_j}\, (f_\varepsilon)_{x_j}+2\,\sum_{i=1}^N g''_{i,\varepsilon}(u_{x_i})\, u_{x_j\,x_i}^2.
\]
By linearity of $L$ we thus get
\[
L\left[|\nabla u|^2\right]= 2\, \sum_{j=1}^N u_{x_j}\, (f_\varepsilon)_{x_j}+2\,\sum_{i,j=1}^N g''_{i,\varepsilon}(u_{x_i})\, u_{x_j\,x_i}^2.
\]
We now fix a pair of concentric balls $B_{r_0}\Subset B_{R_0}\Subset B$ as in the statement of Proposition \ref{lm:berrrnstein}.
Let $\zeta\in C^2_0(B_{R_0})$ be a function such that $0\le \zeta\le 1$ and
\begin{equation}
\label{zeta}
\zeta=1\mbox{ on } B_{r_0},\qquad|\nabla \zeta|^2\le \frac{C}{(R_0-r_0)^2}\, \zeta\qquad \mbox{ and }\qquad |D^2\zeta|\le \frac{C}{(R_0-r_0)^2},
\end{equation}
and consider in $B_{R_0}$ the equation for the function $\zeta\,|\nabla u|^2+\lambda\, u^2$. The crucial parameter $\lambda$ will be chosen later. By using the product rule for $L$ and its linearity, we get
\[
\begin{split}
L\left[\zeta\,|\nabla u|^2+\lambda\, u^2\right]&=\zeta\,L\left[|\nabla u|^2\right]+|\nabla u|^2\,L[\zeta]+2\,\sum_{i,j=1}^N g''_{i,\varepsilon}(u_{x_i})\, (u_{x_j}^2)_{x_i}\,\zeta_{x_i}+\lambda\, L[u^2]\\
&=2\, \sum_{j=1}^N u_{x_j}\, (f_\varepsilon)_{x_j}\,\zeta+2\,\sum_{i,j=1}^N g''_{i,\varepsilon}(u_{x_i})\, u_{x_j\,x_i}^2\,\zeta\\
&+|\nabla u|^2\,L[\zeta]+2\,\sum_{i,j=1}^N g''_{i,\varepsilon}(u_{x_i})\, (u_{x_j}^2)_{x_i}\,\zeta_{x_i}\\
&+2\, u\,\lambda\, L[u]+2\,\lambda\,\sum_{i=1}^N g''_{i,\varepsilon}(u_{x_i})\, u_{x_i}^2.
\end{split}
\]
By using the expression \eqref{operator} of $L$ and the equation \eqref{pointwise},
we can rewrite the previous identity as follows
\begin{equation}
\label{bernstein}
\begin{split}
L\left[\zeta\,|\nabla u|^2+\lambda\, u^2\right]&=2\,\mathcal{I}+2\,\zeta\, \mathcal{G}_1+2\, \mathcal{G}_2+2\,\lambda\, \mathcal{G}_3+\mathcal{G}_4,
\end{split}
\end{equation}
where we used the notation
\[
\mathcal{I}=\lambda\,u\,f_\varepsilon+\zeta\,\sum_{j=1}^N u_{x_j}\,(f_\varepsilon)_{x_j},\qquad \mathcal{G}_1=\sum_{i,j=1}^N g''_{i,\varepsilon}(u_{x_i})\, \left|u_{x_j\,x_i}\right|^2
\]
\[
\mathcal{G}_2=\sum_{i,j=1}^N g''_{i,\varepsilon}(u_{x_i})\, (|u_{x_j}|^2)_{x_i}\,\zeta_{x_i}+\frac{|\nabla u|^2}{2}\,\sum_{i=1}^N g''_{i,\varepsilon}(u_{x_i})\,\zeta_{x_i\,x_i},\qquad\mathcal{G}_3=\sum_{i=1}^N g''_{i,\varepsilon}(u_{x_i})\, |u_{x_i}|^2,
\]
and
\[
\mathcal{G}_4=\sum_{i=1}^N g'''_{i,\varepsilon}(u_{x_i})\, u_{x_i\,x_i}\,\left[2\,\lambda\,u\,u_{x_i}+|\nabla u|^2\, \zeta_{x_i}\right].
\]
We proceed to estimate separately each term on the right-hand side of \eqref{bernstein}.
\vskip.2cm
\centerline{\bf The term $\mathcal{I}$.}
\vskip.2cm
For this, by Young inequality we get
\begin{equation}
\label{I}
\begin{split}
\mathcal{I}&\ge -\lambda\,\|u\|_{L^\infty(B_{R_0})}\,\|f_\varepsilon\|_{L^\infty(B_{R_0})}-N\, \zeta\,|\nabla u|\, |\nabla f_\varepsilon|\\
&\ge -\lambda\,M\,\|f\|_{L^\infty(2\,B)}-\frac{N}{p}\,\zeta\, |\nabla u|^p-\frac{N}{p'}\,\zeta\, \|\nabla f\|_{L^\infty(2\,B)}^{p'},
\end{split}
\end{equation}
where $M$ is the constant appearing in \eqref{max}. We also used that $\|f_\varepsilon\|_{W^{1,\infty}(B_{R_0})}\le \|f\|_{W^{1,\infty}(2\,B)}$.
\vskip.2cm
\centerline{\bf The term $\mathcal{G}_1$.}
\vskip.2cm
This is a positive term and for the moment we simply keep it. It will act as a sponge term, in order to absorb (negative) terms containing the Hessian of $u$.
\vskip.2cm
\centerline{\bf The term $\mathcal{G}_2$.}
\vskip.2cm
This can be estimated by Young inequality and \eqref{zeta} as follows
\[
\begin{split}
\mathcal{G}_2&\ge -\frac{\tau}{2}\,\sum_{i,j=1}^N g''_{i,\varepsilon}(u_{x_i})\, u_{x_j\,x_i}^2\,\zeta_{x_i}^2-\frac{1}{2\,\tau}\,\sum_{i,j=1}^N g''_{i,\varepsilon}(u_{x_i})\,u_{x_j}^2\\
&-\frac{|\nabla u|^2}{2}\,\sum_{i=1}^N g''_{i,\varepsilon}(u_{x_i})\,|\zeta_{x_i\,x_i}|\\
&\ge-\tau\,\frac{C}{2\,(R_0-r_0)^2}\,\zeta\, \mathcal{G}_1-\frac{|\nabla u|^2}{2}\,\left(\frac{C}{(R_0-r_0)^2}+\frac{1}{\tau}\right)\,\sum_{i=1}^N g''_{i,\varepsilon}(u_{x_i}),
\end{split}
\]
where $\tau<1$ is a small positive parameter. We then observe that the last term can be further estimated by using
\begin{equation}
\label{g''}
g''_{i,\varepsilon}(u_{x_i})\le (p-1)\,|\nabla u|^{p-2}+1,
\end{equation}
so that
\[
|\nabla u|^2\, \sum_{j=1}^N g''_{i,\varepsilon}(u_{x_i})\le N((p-1)\, |\nabla u|^p+|\nabla u|^2)\le N\,\left(p-1+\frac{2}{p}\right)\, |\nabla u|^p+N\,\frac{p-2}{p}.
\]
In the end we get
\begin{equation}
\label{G2}
\mathcal{G}_2\ge -\tau\,C_1'\,\zeta\, \mathcal{G}_1-\frac{C'_2}{2\,\tau}\,|\nabla u|^p-\frac{C'_2}{2\,\tau},
\end{equation}
where $C'_1=C_1'(C,R_0-r_0)>0$ and $C'_2=C'_2(p,N,C,R_0-r_0)>0$.
\vskip.2cm
\centerline{\bf The term $\mathcal{G}_3$.}
\vskip.2cm
By using the form of $g_{i,\varepsilon}$, the convexity of the map $m\mapsto m^{p-2}$ and recalling the definition \eqref{delta} of $\delta$, we have
\[
\begin{split}
\mathcal{G}_3=\sum_{i=1}^N g''_{i,\varepsilon}(u_{x_i})\, |u_{x_i}|^2&\ge (p-1)\,\left(\frac{1}{2^{p-3}}\, \sum_{i=1}^N|u_{x_i}|^{p}-\delta^{p-2}\,|\nabla u|^2\right).
\end{split}
\]
By further applying Young inequality to estimate the term $|\nabla u|^2$ and using that
\[
\sum_{i=1}^N |u_{x_i}|^p\ge N^\frac{2-p}{2}\, |\nabla u|^p,
\]
we end up with
\begin{equation}
\label{G3}
\mathcal{G}_3\ge C''_1\, |\nabla u|^p-C''_2,
\end{equation}
where $C''_i=C''(p,N,\delta)>0$, $i=1,2$.
\vskip.2cm
\centerline{\bf The term $\mathcal{G}_4$}
\vskip.2cm
This is the most delicate term and {\it it is precisely here that the condition $p\ge 4$ becomes vital}. First we have
\[
\begin{split}
\mathcal{G}_4\ge -\sum_{i=1}^N |g'''_{i,\varepsilon}(u_{x_i})|\, |u_{x_i\,x_i}|\,\left[2\,\lambda\,|u|\,|u_{x_i}|+|\nabla u|^2\,|\zeta_{x_i}|\right].
\end{split}
\]
Then we observe that by Cauchy-Schwarz inequality (recall the definition \eqref{gepsilon} of $g_{i,\varepsilon}$) we have
\[
\begin{split}
|\nabla u|^2\,\sum_{i=1}^N |g'''_{i,\varepsilon}(u_{x_i})|\,|u_{x_i\,x_i}|\, \,|\zeta_{x_i}|&\le c\,|\nabla u|^2\,\left(\sum_{i=1}^N \,(|u_{x_i}|-\delta_i)_+^{p-2}\,u_{x_i\,x_i}^2\,|\zeta_{x_i}|^2\right)^\frac{1}{2}\\
&\times\left(\sum_{i=1}^N (|u_{x_i}|-\delta_i)_+^{p-4}\right)^\frac{1}{2}\\
&\le c\,\left(\sum_{i=1}^N \,(|u_{x_i}|-\delta_i)_+^{p-2}\,u_{x_i\,x_i}^2\,|\zeta_{x_i}|^2\right)^\frac{1}{2}\,|\nabla u|^\frac{p}{2}
\end{split}
\]
for some constant $c>0$ depending on $p$ and $N$ only. In the last inequality we used that
\[
(|u_{x_i}|-\delta_i)_+^{p-4}\le |u_{x_i}|^{p-4},
\]
since $p\ge 4$.
By further using \eqref{zeta}, the definition of $g_{i,\varepsilon}$ and Young inequality, from the previous inequality we get
\[
|\nabla u|^2\, \sum_{i=1}^N |g'''_{i,\varepsilon}(u_{x_i})|\, |u_{x_i\,x_i}|\,|\zeta_{x_i}|\le c\, \left(\zeta\,\mathcal{G}_1\right)^\frac{1}{2}\, |\nabla u|^\frac{p}{2}\le \tau\,C'''_1\,\zeta\, \mathcal{G}_1+\frac{C'''_1}{\tau}\,|\nabla u|^p,
\]
for some constant $C'''_1=C'''_1(C,N,p,R_0-r_0)>0$.
Similarly, we have
\[
\begin{split}
2\,\lambda\,\sum_{i=1}^N |g'''_{i,\varepsilon}(u_{x_i})|\, |u_{x_i\,x_i}|\,|u|\,|u_{x_i}|&\le c\,\lambda\,M\, \left(\sum_{i=1}^N g''_{i,\varepsilon}(u_{x_i})\,u_{x_i\,x_i}^2 \right)^\frac{1}{2}\,\left(\sum_{i=1}^N  (|u_{x_i}|-\delta_i)_+^{p-4}\,|u_{x_i}|^2\right)^\frac{1}{2}\\
&\le c\, \lambda\,\mathcal{G}_1^\frac{1}{2}\, |\nabla u|^\frac{p-2}{2}\le C'''_2\, \lambda^\frac{2\,p}{p+2}\, \mathcal{G}_1^\frac{p}{p+2}+C_2'''\,|\nabla u|^{p},
\end{split}
\]
for some constant $C'''_2=C_2'''(N,p,M)>0$. By keeping everything together, we get
\begin{equation}
\label{G4}
\mathcal{G}_4\ge -\tau\,C'''_1\,\zeta\, \mathcal{G}_1-\left(\frac{C'''_1}{\tau}+C_2'''\right)\,|\nabla u|^p-C'''_2\, \lambda^\frac{2\,p}{p+2}\, \mathcal{G}_1^\frac{p}{p+2}.
\end{equation}
\vskip.2cm
\centerline{{\bf Collecting all the estimates.}}
\vskip.2cm
We now go back to \eqref{bernstein} and use \eqref{I}, \eqref{G2}, \eqref{G3} and \eqref{G4}. Then we get
\[
\begin{split}
L[\zeta\, |\nabla u|^2+\lambda\, u^2]&\ge \Big[2-\tau\,\big(2\,C_1'+C_1'''\big)\Big]\,\zeta\,\mathcal{G}_1-C'''_2\, \lambda^\frac{2\,p}{p+2}\,\mathcal{G}_1^\frac{p}{p+2}\\
&+\left(2\,\lambda\,C''_1-\frac{C'_2}{\tau}-2\,\zeta\,\frac{N}{p}-\frac{C'''_1}{\tau}-C'''_2\right)\,|\nabla u|^p\\
&-\left(2\,M\,\lambda\, \|f\|_{L^\infty(2\,B)}+2\,\zeta\,\frac{N}{p'}\, \|\nabla f\|_{L^\infty(2\,B)}^{p'}+\frac{C'_2}{\tau}+2\,\lambda\,C''_2\right).
\end{split}
\]
We now choose $\tau$ small enough, in order to make the coefficient of $\zeta\, \mathcal{G}_1$ strictly positive. Then we also choose $\lambda\gg 1$ large enough, so that $|\nabla u|^p$ as well has a strictly positive coefficient. Observe that the choices of $\tau$ and $\lambda$ only depend on the relevant data of the problem and are in particular independent of $\varepsilon$.
By setting for simplicity $\lambda^{2\,p/(p+2)}=\Lambda^2$, this gives
\begin{equation}
\label{quasi}
L[\zeta\, |\nabla u|^2+\lambda\, u^2]\ge \widetilde{c}_1\,\zeta\,\mathcal{G}_1-\widetilde c_2\,\Lambda^{2}\,\mathcal{G}_1^\frac{p}{p+2}+\widetilde{c}_1\, |\nabla u|^p-\widetilde{c}_2,
\end{equation}
where $\widetilde c_1>0$ and $\widetilde c_2>0$ are constants that depend only on $\|f\|_{W^{1,\infty}(2\,B)}, R_0-r_0,\, N,\, p,\, \delta$ and the constant $M$ appearing in \eqref{max}.
\vskip.2cm
Let us now consider the maximum of the function $\zeta\,|\nabla u|^2+\lambda\, u^2$ in $B_{R_0}$. If this maximum is assumed at $x_0\in\partial B_{R_0}$, then we get
\[
\begin{split}
\max_{B_{r_0}} |\nabla u|^2&\le \max_{B_{r_0}} \Big[\zeta\, |\nabla u|^2+\lambda\, u^2\Big]\le \max_{B_{R_0}} \Big[\zeta\, |\nabla u|^2+\lambda\, u^2\Big]\\
&=\zeta(x_0)\, |\nabla u(x_0)|^2+\lambda\, u(x_0)^2\le \lambda\, \|u\|^2_{L^\infty(B_{R_0})}\le \lambda\, M,
\end{split}
\]
thanks to the fact that $\zeta= 1$ on $B_{r_0}$ and $\zeta\equiv 0$ on $\partial B_{R_0}$. This would prove the local Lipschitz estimate. 
\par
In order to conclude, let us now assume that $x_0\in B_{R_0}$, then we get
\[
\nabla\left(\zeta\,|\nabla u|^2+\lambda\, u^2\right)=0\qquad \mbox{ at } x=x_0,
\]
and
\[
D^2\left(\zeta\,|\nabla u|^2+\lambda\, u^2\right)\le 0\qquad \mbox{ at } x=x_0.
\]
Thus at the maximum point $x_0$ we have
\[
L[\zeta\,|\nabla u|^2+\lambda\,u^2]=\sum_{i=1}^N g''_{i,\varepsilon}(u_{x_i})\, \left(\zeta\,|\nabla u|^2+\lambda\,u^2\right)_{x_i\,x_i}\le 0.
\]
By combining this with \eqref{quasi}, we then get
\[
\widetilde{c}_2\ge \widetilde{c}_1\,\zeta\,\mathcal{G}_1-\widetilde c_2\,\Lambda^{2}\,\mathcal{G}_1^\frac{p}{p+2}+\widetilde{c}_1\, |\nabla u|^p.
\]
We multiply the previous by $\zeta(x_0)^{p/2}>0$, then by Young inequality once again we get
\[
\begin{split}
\widetilde{c}_2\,\zeta(x_0)^\frac{p}{2}&\ge \widetilde{c}_1\,\zeta(x_0)^\frac{p+2}{2}\,\mathcal{G}_1-\widetilde c_2\,\Lambda^2\,\left(\zeta(x_0)^\frac{p+2}{2}\,\mathcal{G}_1\right)^\frac{p}{p+2}+\widetilde{c}_1\, \zeta(x_0)^\frac{p}{2}\,|\nabla u(x_0)|^p\\
&\ge \left(\widetilde{c}_1-\widetilde c_2\,\frac{p}{p+2}\,\tau\,\Lambda^2\right)\,\zeta(x_0)^\frac{p+2}{2}\, \mathcal{G}_1-\frac{2\,\tau^{-\frac{p}{2}}}{p+2}\,\Lambda^2+\widetilde{c}_1\, \zeta(x_0)^\frac{p}{2}\,|\nabla u(x_0)|^p.
\end{split}
\]
If we choose
\[
\tau=\frac{p+2}{p}\, \frac{\widetilde{c}_1}{\widetilde c_2}\,\frac{1}{\Lambda^2},
\]
and use that $\zeta\le 1$, we finally get
\begin{equation}
\label{ctilde}
\zeta(x_0)\,|\nabla u(x_0)|^2\le \widetilde{C},
\end{equation}
with $\widetilde{C}=\widetilde{C}(N,p,\delta,\|f\|_{W^{1,\infty}(2\,B)},M,R_0-r_0)>0$. By using this bound we get 
\[
\begin{split}
\max_{B_{r_0}} |\nabla u|^2\le \max_{B_{r_0}} \Big[\zeta\,|\nabla u|^2+\lambda\, u^2\Big]&\le \zeta(x_0)\, |\nabla u(x_0)|^2+\lambda\, u(x_0)^2\\
&\le \widetilde C\, +\lambda\, u(x_0)^2\le \widetilde C+\lambda\,M^2,
\end{split}
\]
which gives the desired conclusion.
\end{proof}

\appendix

\section{Some properties of the functions $g_i$}
\label{sec:A}

The functions $g_i$ have the following convexity property. 
\begin{lm}
\label{lm:convessofuori}
For every $t_1,t_2\in\mathbb{R}$ such that $|t_1-t_2|>2\, \delta_i$ we have
\begin{equation}
\label{stricte}
g_i((1-s)\, t_1+s\, t_2)<(1-s)\,g_i(t_1)+s\, g_i(t_2),\qquad s\in(0,1),\ i=1,\dots,N.
\end{equation}
\end{lm}
\begin{proof}
The function \(g_i\) is convex so that the inequality \(\leq\) holds true for every \(t_1, t_2\). If \(g_i((1-s)\, t_1+s\, t_2)=(1-s)\,g_i(t_1)+s\, g_i(t_2)\), then  \(g_i\) is affine on the segment \([t_1,t_2]\). This can only happen when \(t_1,t_2\in [-\delta_i, \delta_i]\), in which case \(|t_1-t_2|\leq 2\delta_i\).
\end{proof}

They also satisfy the following Lipschitz-type estimate.
\begin{lm}
Let $p\ge 2$. For every $t_1,t_2\in\mathbb{R}$ and $i=1,\dots,N$, we have
\begin{equation}
\label{lipschitz}
|g_i(t_1)-g_i(t_2)|\le \left(|t_1|^{p-1}+|t_2|^{p-1}\right)\, |t_1-t_2|.
\end{equation}
\end{lm}
\begin{proof}
By basic calculus we have
\[
|g_i(t_1)-g_i(t_2)|=|g'_i((1-s)\, t_1+s\, t_2)|\, |t_1-t_2|,
\]
for some $s\in[0,1]$. Since $p\ge 2$, the function $t\mapsto |g'_i(t)|$ is convex and
\[
|g'_i(t)|=(|t|-\delta_i)_+^{p-1}\le |t|^{p-1},\qquad t\in\mathbb{R}.
\] 
Thus we get the conclusion.
\end{proof}
The following basic estimate has been used various times.
\begin{lm}
\label{lm:second_dessous}
Let $p\ge 2$. For every $i=1,\dots,N$ and every $T\ge \delta_i$, we have
\[
g''_{i}(t)\ge (p-1)\,\left(\frac{T-\delta_i}{T}\right)^{p-2}\,\left(T^2+(|t|-T)_+^2\right)^\frac{p-2}{2},\qquad \mbox{ for every } |t|\ge T.
\]
\end{lm}
\begin{proof}
For $T=\delta_i$ there is nothing to prove, thus we can suppose that $T>\delta_i$. We use the elementary inequality 
\[
\frac{T}{T-\delta_i}\, (|t|-\delta_i)\ge |t|, \qquad \mbox{ for every } |t|\ge T.
\]
This implies that for every $|t|\ge T$, we have
\[
\frac{T}{T-\delta_i}\, (|t|-\delta_i)_+\ge T\, +(|t|-T)_+\ge \left(T^2\, +(|t|-T)_+^2\right)^\frac{1}{2}.
\]
By multiplying everything by $(T-\delta_i)/T$ and raising to the power $p-2$, we get the desired conclusion.
\end{proof}

\section{An anisotropic Sobolev inequality in dimension $2$}
\label{sec:B}

In the proof of Proposition \ref{lm:2d} we used Sobolev-Troisi inequality. For the reader's convenience, we give a proof of the particular case we needed.
\begin{lm}
Let $1< q<2$, then for every $u\in C^\infty_0(\mathbb{R}^2)$ we have
\begin{equation}
\label{troisi}
\mathcal{T}_q\,\left(\int_{\mathbb{R}^2} |u|^\frac{4\,q}{2-q}\, dx\right)^\frac{2-q}{2\,q}\le \left(\int_{\mathbb{R}^2} |u_{x_1}|^2\, dx\right)^\frac{1}{2}\,\left(\int_{\mathbb{R}^2} |u_{x_2}|^q\, dx\right)^\frac{1}{q},
\end{equation}
where the constant $\mathcal{T}_q$ is given by
\[
\mathcal{T}_q=\frac{(2-q)^2}{4\,q^2-(2-q)^2}>0.
\]
\end{lm}
\begin{proof}
We first observe that for every $\alpha,\beta>1$, by basic calculus we have
\[
|u(x_1,x_2)|^\alpha = \alpha\, \int_{-\infty}^{x_1} u_{x_1}(t,x_2)\, |u(t,x_2)|^{\alpha-2}\, u(t,x_2)\, dt,
\]
and
\[
|u(x_1,x_2)|^\beta = \beta\, \int_{-\infty}^{x_2} u_{x_2}(x_1,s)\, |u(x_1,s)|^{\beta-2}\, u(x_1,s)\, ds.
\]
Thus
\[
\begin{split}
|u(x_1,x_2)|^{\alpha+\beta}\le \alpha\,\beta\, &\left(\int_{\mathbb{R}} |u_{x_1}(t,x_2)|\, |u(t,x_2)|^{\alpha-1}\, dt\right)\,\left(\int_{\mathbb{R}} |u_{x_2}(x_1,s)|\, |u(x_1,s)|^{\beta-1}\, ds\right).
\end{split}
\]
If we now integrate over $\mathbb{R}^2$ and use Fubini Theorem on the right-hand side, we get
\begin{equation}
\label{depassaggio}
\int_{\mathbb{R}^2} |u|^{\alpha+\beta}\, dx\le \alpha\,\beta\, \left(\int_{\mathbb{R}^2} |u_{x_1}|\, |u|^{\alpha-1}\, dx\right)\,\left(\int_{\mathbb{R}^2} |u_{x_2}|\,|u|^{\beta-1} dx\right).
\end{equation}
By H\"older inequality we then have
\[
\begin{split}
\left(\int_{\mathbb{R}^2} |u_{x_1}|\, |u|^{\alpha-1}\, dx\right)&\left(\int_{\mathbb{R}^2} |u_{x_2}|\,|u|^{\beta-1} dx\right)\\
&\le \left(\int_{\mathbb{R}^2} |u_{x_1}|^2\, dx\right)^\frac{1}{2}\,\left(\int_{\mathbb{R}^2} |u_{x_2}|^q\, dx\right)^\frac{1}{q}\\
&\times \left(\int_{\mathbb{R}^2} |u|^{2\,(\alpha-1)}\, dx\right)^\frac{1}{2} \left(\int_{\mathbb{R}^2} |u|^{\frac{q}{q-1}\,(\beta-1)} dx\right)^\frac{q-1}{q}.
\end{split}
\]
We now choose $\alpha$ and $\beta$ in such a way that
\[
2\,(\alpha-1)=\alpha+\beta\qquad \mbox{ and }\qquad \frac{q}{q-1}\,(\beta-1)=\alpha+\beta,
\]
that is
\[
\alpha=\frac{q+2}{2-q} \quad \mbox{ and }\quad  \beta=\frac{3\,q-2}{2-q}.
\]
Observe that with these choices we have $\alpha+\beta=4\,q/(2-q)$. Thus from \eqref{depassaggio} we get \eqref{troisi}, with
\[
\mathcal{T}_q=\frac{1}{\alpha}\,\frac{1}{\beta}=\frac{(2-q)^2}{4\,q^2-(2-q)^2},
\]
as desired.
\end{proof}

\end{document}